\documentclass[10pt]{amsart}
\usepackage[utf8]{inputenc}
\usepackage[T1]{fontenc}
\usepackage[english]{babel}
\title[Entropy and drift in relatively hyperbolic groups]{Entropy and drift for word metrics on relatively hyperbolic groups}
\author{Matthieu Dussaule and Ilya Gekhtman}

\date{}
\usepackage{silence}
\usepackage{comment}
\usepackage{amsmath}
\usepackage{amssymb}
\usepackage{amsfonts}
\usepackage{amsthm}
\usepackage{tikz}
\usepackage{graphicx}
\usepackage{fancyhdr}
\usepackage{caption}
\usepackage{enumerate}
\usepackage{mathtools}
\usepackage[colorlinks=true,pdfstartview=FitH,bookmarks=false]{hyperref}
\hypersetup{urlcolor=black,linkcolor=black,citecolor=black,colorlinks=true}

\newcommand\N{\mathbb{N}}
\newcommand\Z{\mathbb{Z}}

\newcommand\R{\mathbb{R}}

\theoremstyle{definition}
\newtheorem{hyp}{Assumption}

\theoremstyle{plain}
\newtheorem{definition}{Definition}[section]
\newtheorem{proposition}[definition]{Proposition}
\newtheorem{corollary}[definition]{Corollary}
\newtheorem{theorem}[definition]{Theorem}
\newtheorem{lemma}[definition]{Lemma}
\newtheorem*{thm*}{Theorem}
\newtheorem*{prop*}{Proposition}
\newtheorem*{lem*}{Lemma}

\theoremstyle{remark}
\newtheorem{remark}[definition]{Remark}

\newtheorem*{rem*}{Remark}

\DeclareMathOperator{\Cay}{Cay}
\DeclareMathOperator{\Conv}{Conv}
\DeclareMathOperator{\Diag}{Diag}
\DeclareMathOperator{\supp}{supp}
\DeclareMathOperator{\Tr}{Tr}
\DeclareMathOperator{\diam}{diam}

\begin{document}

\begin{abstract}
    We are interested in the Guivarc'h inequality for admissible random walks on finitely generated relatively hyperbolic groups, endowed with a word metric.
    We show that for random walks with finite super-exponential moment, if this inequality is an equality, then the Green distance is roughly similar to the word distance, generalizing results of Blach\`ere, Ha\"issinsky and Mathieu for hyperbolic groups \cite{BlachereHassinskyMathieu2}.
    Our main applications are for relatively hyperbolic groups with some virtually abelian parabolic subgroup of rank at least 2, relatively hyperbolic groups with spherical Bowditch boundary, and free products with at least one virtually nilpotent factor.
    We show that for such groups, the Guivarc'h inequality with respect to a word distance and a finitely supported random walk is always strict.
\end{abstract}

\maketitle

\section{Introduction}
\subsection{Asymptotic properties of random walks}
Consider a finitely generated group $\Gamma$ together with a probability measure $\mu$ on $\Gamma$.
The $\mu$-random walk starting at the neutral element $e$ of $\Gamma$ is defined as $\omega_n=g_1\cdots g_n$, where $g_k$ are independent random variables distributed as $\mu$.
The law of $\omega_n$ is given by the $n$-th convolution power $\mu^{*n}$ of $\mu$.
We say that such a random walk is admissible if the support of $\mu$ generates $\Gamma$ as a semi-group.
Equivalently, $\bigcup_n\supp \mu^{*n}=\Gamma$.
In other words, the random walk can go everywhere in the group with positive probability.

We will use at some point the following framework for studying random walks.
We consider the map
$$(g_1,...,g_n,...)\in \Gamma^{\N}\mapsto (g_1,g_1g_2,...,g_1g_2...g_n,...)\in \Gamma^{\N}$$
and define the measure $P$ as the push-forward of the product measure $\mu^{\otimes \N}$ under this map.
We refer to $(\Gamma^{\N},P)$ as the path-space and to an element $\omega\in \Gamma^{\N}$ as a sample path for the random walk.
When referring to the random walk, the terminology \emph{almost surely} will be used to mean \emph{for $P$-almost every sample path}.

We say that the probability measure $\mu$ is finitely supported if $\mu(g)=0$ except for finitely many $g$.
We also say that $\mu$ has finite exponential (respectively super-exponential) moment with respect to a left invariant metric $d$ on $\Gamma$ if
$$\sum_{g\in \Gamma}\mu(g)c^{d(g,e)}<\infty$$ for some (respectively all) $c>1$.
We say it has finite first moment if $$\sum_{g\in \Gamma}\mu(g)d(g,e)<\infty.$$
We will be mainly interested in probability measures with finite super-exponential moment.

In the following, the groups we will consider are non-amenable.
In particular, any admissible random walk on such a group $\Gamma$ is transient, meaning that it almost surely goes back to $e$ only a finite number of time.
Equivalently, it almost surely visits every element in the group a finite number of time.
In this setting, the random walk almost surely goes to infinity.
It has been a fruitful line of research to understand the behaviour at infinity of $\omega_n$ in terms of geometric large scale properties of the group,
see \cite{Kaimanovich2}, \cite{Kaimanovich}, \cite{Vershik}, \cite{Woess}.
One particular question one can ask is whether the law of $\omega_n$ can be compared with the uniform law on large balls for a word distance on $\Gamma$.
Let us give a precise formulation of this question.

We first recall the definitions of some asymptotic quantities associated with the random walk and the group.
We fix a finite generating set $S$ for $\Gamma$ and denote by $d_{w,S}$ the associated word distance.
When the generating set $S$ is implicit, we will denote it by $d_w$.
For simplicity, we will also use the notations $\|g\|_S=d_{w,S}(e,g)$ for $g\in \Gamma$ and $\|g\|=d_w(e,g)$ when $S$ is implicit.
Define the entropy of $\mu$ as
$$H(\mu)=-\sum_{g\in \Gamma}\log (\mu(g))\mu(g)$$
and the drift of $\mu$ as
$$L(\mu)=\sum_{g \in \Gamma}\mu(g)\|g\|_S.$$
Those two quantities can be finite or infinite.
Note that finiteness of $H(\mu)$ is implied by finiteness of $L(\mu)$, see for example \cite{Derriennic}.
Both quantities are sub-additive with respect to the convolution power, that is,
$$H(\mu*\nu)\leq H(\mu)+H(\nu), L(\mu*\nu)\leq L(\mu)+L(\nu).$$
In particular, the sequences $\frac{H(\mu^{*n})}{n}$ and $\frac{L(\mu^{*n})}{n}$ have respective limits $h_{\mu}$ and $l_{\mu,S}$.
We call $h_{\mu}$ the asymptotic entropy and $l_{\mu,S}$ the asymptotic drift of the random walk with respect to $S$ and simply denote them by $h$ and $l$ when not ambiguous.
Moreover, according to Kingman's ergodic theorem \cite[Theorem~10.1]{Walters} when $H(\mu)$ is finite,
almost surely, one has
$$h=\lim \frac{-\log \mu^{*n}(\omega_n)}{n}.$$
Similarly, when $L(\mu)$ is finite, then, almost surely, one has
$$l=\lim \frac{\|\omega_n\|}{n}.$$

We also define another quantity as follows.
Denote by $b_n$ the cardinality of the ball of center $e$ and radius $n$ for the word metric $d_w$.
Then, the sequence $b_n$ is sub-multiplicative, so that $\frac{\log b_n}{n}$ converges to some limit $v_S$.
We call $v_S$ the volume growth of $\Gamma$ with respect to $S$ and simply denote it by $v$ when not ambiguous.
This quantity only depends on the group and not on the random walk.

\medskip
The fundamental inequality of Guivarc'h (see \cite{Guivarch}) states that
\begin{equation}\label{Guivarchinequality}
    h\leq lv.
\end{equation}
One can interpret this inequality as follows.
On the one hand, the random walk at time $n$ essentially lies in a set of cardinality $\mathrm{e}^{hn}$ (see for example \cite[Proposition~1.13]{Haissinsky} for a precise formulation).
On the other hand, it is asymptotically contained in the ball of radius $ln$, which has cardinality $\mathrm{e}^{lvn}$.
Thus, $\mathrm{e}^{hn}\leq \mathrm{e}^{lvn}$ for large $n$, hence $h\leq lv$.

Thus, comparing the law of $\omega_n$ with a uniform law on large balls is in some weak sense similar to asking whether $h=lv$ or $h<lv$.
As we will see below, in the context of relatively hyperbolic groups, there is a strong connection between the two questions.

\subsection{Probabilistic and geometric boundaries}
Another way of comparing the asymptotic properties of $\omega_n$ with large scale geometric properties of $\Gamma$ is to compare probabilistic boundaries with geometric boundaries.
Let us give some details.
We define the Green function of the random walk as
$$G(g,g')=\sum_{n\geq 0}\mu^{*n}(g^{-1}g').$$
It is left $\Gamma$-invariant, that is $G(g,g')=G(g''g,g''g')$ for every $g,g',g''$.
This function encodes a lot of properties of $\mu^{*n}$.

It satisfies that
$$G(g,g')=P(g\to g')G(e,e),$$
where $P(g\to g')$ is the probability that the random walk ever reaches $g'$, starting at $g$ (see \cite[Lemma~1.13.(b)]{Woess}).
We define the Green distance as
$$d_G(g,g')=-\log G(g,g')+\log G(e,e).$$
Thus, $d_G(g,g')=-\log P(g\to g')$.
This distance was introduced by Blach\`ere and Brofferio in \cite{BlachereBrofferio}.
When the measure $\mu$ is symmetric and the random walk is transient, it is truely a distance.
We will still call it the Green distance in general, even when $\mu$ is not symmetric.

The \emph{Martin compactification} of the random walk is the horofunction compactification of $\Gamma$ for the Green distance.
It is thus the smallest compact set $M$ such that the Martin kernel $K(\cdot,\cdot)$ defined as
$$K(g,g')=\frac{G(g,g')}{G(e,g')}$$
extends continuously as a function on $\Gamma\times M$.
The \emph{Martin boundary} $\partial_{\mu}\Gamma$ is the complement of $\Gamma$ in this compactification.
It always abstractly exists (see \cite{Sawyer}), and since $\Gamma$ preserves the Green distance it acts by homeomorphisms on the Martin boundary.
However, its identification in terms of the geometry of $\Gamma$ is often a difficult problem.

The random walk almost surely converges to some point $\xi$ in the Martin boundary.
One can thus define a measure on $\partial_{\mu}\Gamma$, which is the law of the exit point $\xi$, see \cite[Section~3]{Sawyer} for more details.
This is the \emph{harmonic measure}, that we will denote by $\nu$.
More generally, if the random walk starts at some point $g$ in $\Gamma$ (that is $\omega_n$ is replaced with $g \omega_n$), one can still define the exit point in $\partial_{\mu}\Gamma$ and the corresponding harmonic measure $\nu_{g}$.
The Martin boundary endowed with the harmonic measure is a model for the so-called \emph{Poisson boundary}.
We refer to \cite{KaimanovichVershik} and \cite{Kaimanovich} for many more details and other equivalent definitions, as well as \cite{Kaimanovich2} for more details on the relations between the Martin and the Poisson boundaries.

\medskip
One can also define the horofunction boundary for the word distance.
However, this boundary is typically too large for geometric applications
and there is not always a direct relation between probabilistic boundaries and the horofunction boundary.
In some groups, one can define other boundaries which better capture the geometry of the group.
This is in particular the case for hyperbolic groups, where one can define the Gromov boundary.
The relation between the Poisson boundary and the Gromov boundary was explored by Kaimanovich among others, see in particular \cite{Kaimanovich}.
Kaimanovich showed that for random walks with finite first moment the Gromov boundary endowed with the unique stationary measure is a realization of the Poisson boundary (see \cite{Kaimanovich} for weaker assumptions).

Ancona \cite{Ancona} also proved that for finitely supported $\mu$, the Martin boundary is homeomorphic to the Gromov boundary.
This is a stronger statement than identifying the Gromov boundary with the Poisson boundary.
Indeed, the identification with the Poisson boundary gives an isomorphism of measured spaces (in particular, up to some set of measure 0), whereas the identification with the Martin boundary gives a homeomorphism.
The proof of Ancona relies on the following crucial estimates, the \emph{Ancona inequalities}, a generalization of which we will use in this paper.
These inequalities state that there exists $C\geq 1$ such that the following holds.
If $x,y,z$ are three points in $\Gamma$ along a word-geodesic (in this order), then
$$\frac{1}{C}G(x,y)G(y,z)\leq G(x,z)\leq CG(x,y)G(y,z).$$
One can interpret these inequalities saying that the random walk has to go through $y$ to go from $x$ to $z$ with probability bounded from below.
In general, if there exists $C>0$ such that two quantities $f$ and $g$ satisfy that $\frac{1}{C}f\leq g\leq C f$, we will use the notation
$f\asymp g$.
If the notation is ambiguous, for example if the constant $C$ depends on some parameters, we will avoid using this notation, except if the dependence is clear from the context.
Also, whenever $f\leq C g$ for some constant $C$, we will use the notation $f\lesssim g$.
We can thus restate Ancona inequalities as
\begin{equation}\label{Anconainequality}
    G(x,z)\asymp G(x,y)G(y,z)
\end{equation}
whenever $y$ is on some word-geodesic from $x$ to $z$.

\medskip
In the context of non-elementary hyperbolic groups, one can interpret the Guivarc'h inequality~(\ref{Guivarchinequality}) as follows.
One can endow the Gromov boundary with a visual distance, which depends on some parameter $\epsilon$.
Then, the Hausdorff dimension of the Gromov boundary with the visual distance is equal to $\frac{v}{\epsilon}$, whereas the Hausdorff dimension of the harmonic measure $\nu$ is given by $\frac{h}{\epsilon l}$, see \cite{Ledrappier} for the case of the free group and \cite{BlachereHassinskyMathieu1} and \cite{BlachereHassinskyMathieu2} for more details in the general case.
We thus recover that $\frac{h}{l}\leq v$.
Moreover, $h=lv$ if and only if the harmonic measure $\nu$ has maximal dimension in the Gromov boundary.

Blach\`ere, Ha\"issinsky and Mathieu proved in \cite{BlachereHassinskyMathieu2} that this is a very rigid condition.
They compare the harmonic measure $\nu$ to Patterson-Sullivan measures, which are quasi-conformal measures defined on the Gromov boundary and having maximal Hausdorff dimension for the visual distance.
They prove the following.
Let $\Gamma$ be a non-elementary hyperbolic group endowed with a word metric $d_w$.
Let $\mu$ be a symmetric admissible finitely supported probability measure on $\Gamma$.
Then,
$h=lv$ if and only if the Patterson-Sullivan and harmonic measures are equivalent, if and only if the Green distance is roughly similar to the word metric (precisely, $|d_G-vd_w|\leq C$ for some constant $C$), see \cite[Theorem~1.5]{BlachereHassinskyMathieu2}.
As we will see, the symmetry assumption on $\mu$ is not needed.
Gou\"ezel, Math\'eus and Maucourant then used this result to prove in \cite{GMM} that in this context, $h=lv$ can only occur if the group is virtually free.

\subsection{Random walks in relatively hyperbolic groups}
In this paper, we are interested in relatively hyperbolic groups.
Following Bowditch \cite{Bowditch}, a finitely generated group is called hyperbolic relative to a collection of subgroups $\Omega$ if it acts properly discontinuously and by isometries on a proper geodesic hyperbolic metric space $X$ such that every limit point is either conical or bounded parabolic. Moreover, the stabilizers of the parabolic limit points are exactly the elements of $\Omega$.
We will give more details in Section~\ref{Sectionbackgroundrelhyp}.

A relatively hyperbolic group is equipped with a boundary, called the Bowditch boundary, which is the limit set of $\Gamma$ in the Gromov boundary of a space $X$ on which $\Gamma$ acts as in the definition.
Once $\Omega$ is fixed, the Bowditch boundary does not depend on $X$ up to homeomorphism.
We will denote by $\partial_B\Gamma$ the Bowditch boundary of a relatively hyperbolic group $\Gamma$.
Note that it coincides with the Gromov boundary for hyperbolic groups.

There are many examples of relatively hyperbolic groups.
The combinatorial archetype is given by a free product of finitely many finitely generated groups,
which is hyperbolic relative to its free factors.
Another example is given by geometrically finite Kleinian groups, which are hyperbolic relative to the stabilizers of the cusps of the corresponding geometrically finite hyperbolic manifold.
In this case, the parabolic subgroups are virtually abelian and the hyperbolic space $X$ in Bowditch's definition can be taken to be the real hyperbolic space $\mathbb{H}^n$.
More generally, the fundamental group of a geometrically finite manifold with pinched negative curvature (see \cite{Bowditch2}) is hyperbolic relative to the stabilizers of the cusps.
In this situation, parabolic subgroups are virtually nilpotent.

\medskip
Random walks on relatively hyperbolic groups have been studied by many people.
The Poisson boundary can be identified with the Bowditch boundary endowed with the unique stationary measure $\nu$, as soon as $\mu$ has finite first moment (again, see \cite{Kaimanovich} for weaker assumptions).
This is in particular true for measures $\mu$ with finite super-exponential moment, which are the measures we consider in this paper.
This stationary measure $\nu$ is non-atomic and thus gives full measure to conical limit points.
This was proved by Kaimanovich for Kleinian groups \cite[Section~9]{Kaimanovich}.
In general, one can define the coned-off graph $\hat{\Gamma}$ (see \cite{Farb}), which is a (non-proper) hyperbolic graph on which $\Gamma$ acts \emph{acylindrically} by isometries (see \cite{Osin} for more details) and whose Gromov boundary $\partial \hat{\Gamma}$ is the set of conical limit points, with parabolic points lying in its interior.
This action on the coned-off graph allows us to identify the Poisson boundary with $\partial \hat{\Gamma}$ and we recover the fact that the set of conical limit points with the stationary measure is the Poisson boundary (see \cite[Section~6]{MaherTiozzo} for a more general result).
Note also that according to results of Karlsson \cite{Karlsson2}, the Poisson boundary can be identified with the support of a measure on the Floyd boundary, which is a geometric boundary that we will define below.
The Floyd boundary always covers the Bowditch boundary and the pre-image of conical limit points is a single point according to results of Gerasimov \cite{Gerasimov}.

The precise identification of the Martin boundary up to homeomorphism is much more difficult.
In general, we do not know the Martin boundary of a finitely supported admissible measure $\mu$ on a relatively hyperbolic group.
However, Gekhtman-Gerasimov-Potyagailo-Yang proved in \cite{GGPY} that the Martin boundary always covers the Bowditch boundary (it actually covers the Floyd boundary).
Moreover, the pre-image of a conical limit point consists of one point.
We thus get an identification of the set of conical limit points as a subset of the Martin boundary. The results of \cite{GGPY} in fact hold  for measures with super-exponential first moment.

When the parabolic subgroups are virtually abelian and the measure is finitely supported, the whole Martin boundary is described up to homeomorphism in \cite{DGGP}.
The parabolic limit points in the Bowditch boundary have to be blown-up into spheres of the appropriate dimension.
For Kleinian groups, the Martin boundary coincides with the CAT(0) boundary of the group.
We will use this identification below and we will give a precise statement then (see Theorem~\ref{convergenceinMartin}).

\medskip
Our first theorem is an extension to relatively hyperbolic groups of the results of Blach\`ere, Ha\"issinsky and Mathieu \cite{BlachereHassinskyMathieu2} described above. The major technical difficulty is that the word metric we are considering is not Gromov hyperbolic.
Classical techniques of boundary dynamics in coarse negative curvature are not available. In their place, we use the theory of Floyd functions on relatively hyperbolic groups developed by Gerasimov and Potyagailo \cite{Gerasimov} \cite{GePoJEMS} \cite{GePoCrelle} and the analogue in that setting of Patterson-Sullivan measures developed by Yang \cite{Yang}
We refer to Sections~\ref{SectionrelhypandFloyd} and~\ref{Sectionquasiconformality} for more details.
Among the technical difficulties is the following.
In the hyperbolic setting of \cite{BlachereHassinskyMathieu2} and \cite{GMM}, the authors crucially use the fact that harmonic measures are doubling with respect to a natural metric on the boundary, allowing them to apply the Lebesgue differentiation theorem.
This is not known to be true in our setting.
However, we show that the measures satisfy a weaker regularity property with respect to partial shadows, allowing us to use a differentiation theorem from \cite{Federer}.

\begin{theorem}\label{theoremBHM}
Let $\Gamma$ be a non-elementary relatively hyperbolic group.
Let $\mu$ be an admissible probability measure on $\Gamma$ with finite super-exponential moment with respect to a word metric $d_w$ on $\Gamma$.
Let $d_G$ be the Green metric associated with $\mu$.
Denote by $\nu$ the harmonic measure on the Bowditch boundary $\partial_B\Gamma$ with respect to $\mu$.
Denote by $\kappa$ the Patterson-Sullivan measure based at $e$ on $\partial_B\Gamma$ with respect to $d_w$.
Finally, denote respectively by $h,l,v$ the asymptotic entropy, drift and volume growth associated with $\mu$ and $d_w$.
The following conditions are equivalent.
\begin{enumerate}
    \item The equality $h=lv$ holds.
    \item The measure $\nu$ is equivalent to the measure $\kappa$.
    \item The measure $\nu$ is equivalent to the measure $\kappa$ with Radon-Nikodym derivatives bounded from above and below.
    \item There exists $C\geq 0$ such that for every $g \in \Gamma$, $|d_G(e,g)-vd_w(e,g)|\leq C$.
\end{enumerate}
\end{theorem}
\begin{remark}
A result analogous to Theorem \ref{theoremBHM} can be also be proved where instead of the word metric we take a metric on $\Gamma$ coming from its action on the hyperbolic space $X$. The arguments in this "geometric metric" case are considerably simpler because one has to deal with the geometry of the hyperbolic space $X$ instead of the Cayley graph and the relevant results about Patterson-Sullivan measures have been proved by Coornaert \cite{Coornaert}. For symmetric measures, the equivalence of the items analogous to (2), (3), (4) above were proved in Section 11 of \cite{GGPY}. A complete proof will be a special case of a result concerning Gibbs measures in \cite{GTgibbs}.
\end{remark}
We do not interpret this result in terms of Hausdorff dimensions.
Note however that $v$ is, up to a multiplicative constant (that depends on some parameter $\lambda$), the Hausdorff dimension of the Bowditch boundary for a well-defined distance.
This distance is the pull-back of the Floyd distance, which is defined on the Floyd boundary (see \cite[Theorem~1.1, Theorem~1.2]{PotyagailoYang} for more details, note in particular that the parameter $\lambda$ above actually depends on the choice of the Floyd distance).
In the particular case of free products, Candellero, Gilch and M\"uller proved in \cite{CandelleroGilchMuller} that $v$ is the Hausdorff dimension of the set of infinite words, seen as a subset of the the space of ends for some visual distance on it.
Also, Tanaka proved that for measures with finite first moment, $h/l$ is the Hausdorff dimension of the harmonic measure with respect to a visual distance coming from the Gromov product on the coned-off graph $\hat{\Gamma}$, see \cite[Theorem~1.3]{Tanaka} for more details.

\subsection{Strict inequality in the Guivarc'h inequality}
We use this rigidity result to prove that $h<lv$ in several cases.

\medskip
Our first application is to relatively hyperbolic groups with virtually abelian parabolic subgroups.
The description of the Martin boundary in \cite{DGGP} allows us to prove that $h<lv$ for finitely supported measures $\mu$, if one of the parabolic subgroups  is virtually abelian of rank at least 2.

\begin{theorem}\label{theoremhlvabelianparabolics}
Let $\Gamma$ be a non-elementary relatively hyperbolic group.
If one of the parabolic subgroups is virtually abelian of rank at least 2, then $h<lv$ for any word distance and any finitely supported admissible measure $\mu$.
\end{theorem}

Using the results of Gou\"ezel, Math\'eus and Maucourant \cite{GMM}, we then get the following corollary.

\begin{corollary}\label{coroabelianparabolics+GMM}
Let $\Gamma$ be a non-elementary relatively hyperbolic group with virtually abelian parabolic subgroups.
If $\mu$ is an admissible finitely supported measure such that $h=lv$ for some word distance, then $\Gamma$ is virtually free.
\end{corollary}

We also prove a theorem for free products.
Let $\Gamma=\Gamma_1*\cdots *\Gamma_N$ be a free product of finitely generated groups.
We choose a finite generating set $S_i$ for each free factor $\Gamma_i$ and define $S=\cup S_i$.
Then $S$ is a finite generating set for $\Gamma$.
Such a generating set is called \emph{adapted}.
We also say that the word metric is adapted.

\begin{theorem}\label{hlvnilpotentfreefactors}
Let $\Gamma=\Gamma_1*\cdots*\Gamma_N$ be a free product of finitely generated groups.
Assume that $\Gamma_1$ is a non-virtually cyclic nilpotent group.
Then $h<lv$ for any adapted word metric and any admissible probability measure $\mu$ with finite super-exponential moment.
\end{theorem}


Finally, we give a result for relatively hyperbolic groups whose Bowditch boundary are spheres.

\begin{theorem}\label{sphericalBowditchboundary}
Let $\Gamma$ be a non-elementary relatively hyperbolic group.
Assume that the Bowditch boundary is homeomorphic to the $d$-sphere $\mathbb{S}^{d}$, with $d\geq 2$.
Then, $h<lv$ for any word distance and any admissible measure $\mu$ with finite super-exponential moment.
\end{theorem}

We can actually prove a stronger statement, which is a bit more technical to state (see Theorem~\ref{mainTheoremconnected}). 
However, this one is sufficient to get the following corollary.

\begin{corollary}\label{corofundamentalgroups}
Let $M$ be a finite volume Riemannian manifold of pinched negative curvature and of dimension $n\geq 3$.
Then, $h<lv$ for any word metric on $\pi_1(M)$ and any admissible measure $\mu$ on $\pi_1(M)$ with finite super-exponential moment.
\end{corollary}

For hyperbolic manifolds and finitely supported random walks, this corollary is also a consequence of Theorem~\ref{theoremhlvabelianparabolics}.
However, in general, parabolic subgroups are not necessarily virtually abelian, but as explained, they are virtually nilpotent.

We remark that unlike \cite{GMM} and \cite{BlachereHassinskyMathieu2}, we do not assume the measure $\mu$ is symmetric in any of our results.




\subsection{Organization of the paper}
Let us now give some more details on the organization of the paper and on the proofs of our results.

\medskip
In Section~\ref{Sectionbackgroundrelhyp}, we prove some geometric properties of relatively hyperbolic groups and recall the construction of partial shadows from \cite{Yang}. We also state a version of the Lebesgue differentiation theorem for measures on the Bowditch boundary that satisfy a weak doubling property with respect to the partial shadows, see Proposition~\ref{localsingularity} for a precise statement. Section~\ref{SectionLebesguetheorem} is devoted to the proof of this Lebesgue differentiation theorem.

\medskip
We begin Section~\ref{Sectionrigidity} recalling some facts from \cite{Yang} about Patterson-Sullivan measures $\kappa_{g}$ constructed with the word metric in relatively hyperbolic groups.
In particular, we recall the shadow lemma, which states that $\kappa_e(\Omega(g))$ is approximately $\mathrm{e}^{-vd_w(e,g)}$,
where $\Omega(g)$ is a partial shadow at $g$ (see Sections~\ref{SectionrelhypandFloyd} and~\ref{Sectionquasiconformality} for the precise definitions).
We also prove a shadow lemma for the harmonic measure $\nu$.
We show that $\nu(\Omega(g))$ is approximately $\mathrm{e}^{-d_G(e,g)}$.
This second shadow lemma is based on the following relative Ancona inequalities proved in \cite{GGPY}.
A transition point on a geodesic is a point that is not deep inside a parabolic subgroup (see the precise definition in Section~\ref{Sectionbackgroundrelhyp}).
Let $x,y,z\in \Gamma$ and assume that $y$ is at distance at most $r$ from an $(\epsilon,\eta)$-transition point on a geodesic from $x$ to $z$.
Then
\begin{equation}\label{relativeAncona}
    G(x,z)\asymp G(x,y)G(y,z)
\end{equation}
where the implicit constant only depends on $\epsilon$, $\eta$ and $r$.
We use these estimates together with our Lebesgue differentiation theorem as tools to prove Theorem~\ref{theoremBHM}, following a strategy inspired by Blach\`ere, Ha\"issinsky and Mathieu \cite{BlachereHassinskyMathieu2}.

\medskip
In Section~\ref{Sectionabelianparabolics}, we first prove that if $h=lv$ in a relatively hyperbolic group, then Ancona inequalities like~(\ref{Anconainequality}) hold for every word geodesic in the group.
Recall that these inequalities basically state that to go from a point $x$ to a point $z$ in $\Gamma$, the random walk has to visit the points $y$ on a geodesic from $x$ to $z$ with probability bounded from below.
This property certainly is not true in virtually abelian groups of rank at least 2.
It even seems reasonable to think that if Ancona inequalities hold everywhere in the group, then the group has to be hyperbolic.
However, we cannot prove such a statement, the problem being that the parabolic subgroup are not isometrically embedded and by replacing geodesics with quasi-geodesics, we lose information.
We can still find a contradiction when parabolic subgroups are virtually abelian of rank at least 2.
We use the description of the Martin boundary given in \cite{DGGP} as well as estimates for the Green function restricted to the parabolic subgroups proved in \cite{Dussaule}.
These contradict Ancona inequalities and we
can then prove Theorem~\ref{theoremhlvabelianparabolics}.

\medskip
In Section~\ref{Sectionnilpotentparabolics}, we are interested in virtually nilpotent parabolic subgroups.
As explained, our proof of the fact that $h=lv$ cannot hold when the parabolic subgroups are virtually abelian uses crucial estimates for the Green function restricted to the parabolic subgroups.
Such estimates are not known when the parabolic subgroups are not virtually abelian.
Instead, we prove the following property.
If (for any random walk on a finitely generated group $\Gamma$) the Green and word metrics are roughly similar, then the number of middle points on geodesics between two points $x,y\in \Gamma$ is uniformly bounded independently of $x$ and $y$ (see Proposition~\ref{propfinitemiddles}).
To find a contradiction and prove that $h<lv$, we thus have to construct lots of different geodesics with different middle points from one point to another.
Such a construction is possible in nilpotent groups, using the work of Walsh \cite{Walsh}.
We can then construct geodesics in free products with an adapted word distance when one of the free factor is nilpotent.
We deduce from all this Theorem~\ref{hlvnilpotentfreefactors}.
It seems conceivable (although we could not prove this) that the uniform bound on middle points of geodesics between two points implies hyperbolicity of the group. If this were the case, we would be able to conclude that for any finite super-exponential first moment random walk on a relatively hyperbolic group, $h=lv$ implies  $\Gamma$ is virtually free. 
\medskip

Section~\ref{Sectionconnectedconicallimitpoints} deals with connectedness properties of the different boundaries of $\Gamma$.
In particular, we relate these boundaries to the stable translation spectrum of the Green and word metric and we use them to prove Theorem~\ref{sphericalBowditchboundary}.
More precisely, we show that the stable translation spectrum of any word metric on a non-elementary relatively hyperbolic group is arithmetic: the stable translation length for $d_w$ of a loxodromic element is rational with bounded denominator.
This generalizes known results for hyperbolic groups.
When the Bowditch boundary is a topological sphere of dimension $\geq 2$ (or more generally, the conical limit set minus any two points has a finite number of connected components), we show that the stable translation length spectrum of the Green metric is not arithmetic. This uses continuity of cross-ratios for the Green metric on the conical limit set, which we prove using the relative Ancona inequalities from \cite{GGPY}.
Since arithmeticity of the stable translation spectrum is preserved under rough similarity of metrics, we are able to conclude Theorem \ref{sphericalBowditchboundary}.
Along the way, we prove a technical result of independent interest: the Gromov product for the Green metric, which is called the \textit{Naim kernel} in this context, extends continuously to pairs of conical points on the Bowditch boundary (a set with countable complement), see Proposition \ref{continuityNaim}.




\section{Background on relatively hyperbolic groups}\label{Sectionbackgroundrelhyp}
\subsection{Relatively hyperbolic groups and the Floyd metric}\label{SectionrelhypandFloyd}
Let $\Gamma$ be a finitely generated group. The action of $\Gamma$ on a compact Hausdorff space $T$ is called a convergence action if the induced action on triples of distinct points of $T$ is properly discontinuous. 
Suppose $\Gamma\curvearrowright T$ is a convergence action. The set of accumulation points  $\Lambda \Gamma$ of any orbit $\Gamma \cdot x\ (x\in T)$ is called the {\it limit set} of the action. As long as $\Lambda \Gamma$ has more than two points, it is uncountable and is the unique minimal closed $\Gamma$-invariant subset of $T$.
The action is then said to be non-elementary. In this case, the orbit of every point in $\Lambda \Gamma$ is infinite. 
The action is {\it minimal} if $\Lambda \Gamma=T$.

{A point $\zeta\in\Lambda \Gamma$  is called {\it conical} if there is a sequence $(g_{n})$ of $\Gamma$ and distinct points $\alpha,\beta \in \Lambda \Gamma$ such that
$g_{n}\zeta \to \alpha$ and $g_{n}\eta \to \beta$ for all $\eta \in  T \setminus\{\zeta\}.$}
The point $\zeta\in\Lambda \Gamma$ is called bounded parabolic if it is the unique fixed point of its stabilizer in $\Gamma$, which is infinite and acts cocompactly on $\Lambda \Gamma \setminus \{\zeta\}$.
The stabilizers of bounded parabolic points are called (maximal) parabolic subgroups.
The convergence action $\Gamma \curvearrowright  T$ is called geometrically finite if every point of $\Lambda \Gamma \subset T$ is either conical or bounded parabolic.
Since $\Gamma$ is assumed to be finitely generated, every maximal parabolic subgroup is finitely generated too (see \cite[Main~Theorem~(d)]{Gerasimov2}).
Then, by Yaman's results \cite{Yaman}, it follows that if $\Gamma \curvearrowright T$ is a minimal geometrically finite action, then there exists a proper geodesic Gromov hyperbolic space $X$ on which $\Gamma$ acts properly discontinuously by isometries and a $\Gamma$-equivariant homeomorphism $T \to \partial X$.

Suppose now $\Omega$ is a collection of subgroups of $\Gamma$. We say that $\Gamma$ is \textit{hyperbolic relative to $\Omega$} if there exists some compactum $T$ on which $\Gamma$ acts minimally and geometrically finitely and such that the maximal parabolic subgroups are the elements of $\Omega$.
Such a compactum is then unique up to $\Gamma$-equivariant homeomorphism \cite{Bowditch} and is called the Bowditch boundary of $(\Gamma, \Omega)$. 
The group $\Gamma$ is non-elementary relatively hyperbolic if it admits a non-elementary geometrically finite convergence action on some infinite compactum.
A useful fact is the following. Let $\Gamma$ be a group hyperbolic relative to a collection of parabolic subgroups $\Omega$.
The set $\Omega$ is invariant under conjugacy, since the set of parabolic limit points is invariant under translation.
Furthermore, the set $\Omega$ contains at most finitely many conjugacy classes of maximal parabolic subgroups (see \cite[Theorem~1B]{Tukia}).

Let $f:\mathbb{R}^{+}\to \mathbb{R}^{+}$ be a function satisfying two conditions: $\sum_{n\geqslant0}f_n<\infty$ and there exists a  $\lambda\in (0,1)$ such that $1\geqslant f_{n+1}/f_n\geqslant\lambda$ for all $n\in\mathbb{N}$. The function $f$ is called the {\it rescaling function.}
Pick a basepoint $o\in \Gamma$ and rescale the Cayley graph $\Cay(\Gamma,S)$ by declaring the length of an edge $\sigma$ to be $f(d(o,\sigma))$. The induced shortpath metric on $\Cay(\Gamma,S)$  is called the {\it Floyd metric} with respect to the basepoint $o$ and Floyd function $f$ and denoted by $\delta^{f}_{o}(.,.)$.
Its Cauchy completion (whose topology does not depend on the basepoint) is called the Floyd compactification $\overline{\Gamma}_{f}$ and $\partial_{f}\Gamma= \overline{\Gamma}_{f} \setminus \Gamma$ is called the Floyd boundary.

On the one hand, Karlsson showed that the action of a group on its Floyd boundary is always a convergence action \cite[Theorem~2]{Karlsson}.
On the other hand, if $\Gamma$ is relatively hyperbolic and if the Floyd function $f$ is not decreasing exponentially too fast, Gerasimov proved that there is a continuous $\Gamma$-equivariant surjection ({\it Floyd map}) from the Floyd boundary to the Bowditch boundary \cite[Map theorem]{Gerasimov}.
Furthermore,  Gerasimov and Potyagailo \cite[Theorem~A]{GerasimovPotyagailo} proved that the pre-image  of any conical point by this map is a singleton and the pre-image of a parabolic fixed point $p$ is the limit set for the action of its stabilizer $\Gamma_p$ on  $\partial_f\Gamma$.
In particular if $\Gamma_p$ is an amenable non-virtually cyclic group then its limit set on the Floyd boundary is a point. Consequently, when $\Gamma$ is hyperbolic relative to a collection of infinite amenable subgroups which are not virtually cyclic, the Floyd boundary is homeomorphic to the Bowditch boundary.

The Floyd metric extends to a bounded metric of the same name on the Floyd compactification $\overline{\Gamma}_{f}$.
In turn, this can be transferred to the \emph{shortcut metric} on the Bowditch compactification $\overline{\Gamma}_{B}=\Gamma \cup \partial_{B}\Gamma$ as follows.
Let $\overline{\delta}^{f}_{o}$ be the maximal pseudo-metric on $\overline{\Gamma}_{f}$ which is constant on fibers of Gerasimov's map  $F:\overline{\Gamma}_{f}\to \overline{\Gamma}_{B}$ and satisfies $\overline{\delta}^{f}_{o}\leq \delta^{f}_{o}$.
Then $\overline{\delta}^{f}_{o}(F(.),F(.))$ defines a true metric on the Bowditch compactification, called the shortcut metric, which we also denote by $\overline{\delta}^{f}_{o}$, see \cite[Definition~2.6]{GePoJEMS} for more details on the construction.
The Floyd and shortcut metrics are Lipschitz with respect to change in basepoint: there is a $\lambda>0$ such that 
$\delta^{f}_{o}/\delta^{f}_{o'}\leq \lambda^{d_{w}(o,o')}$ for $o,o'\in \Gamma$ and the same holds for $\overline{\delta}^{f}_{o}$. 
See \cite[Section 2]{GePoJEMS} for details.

\begin{definition}
If $\alpha$ is a (finite or infinite) geodesic in $\Cay(\Gamma,S)$ for the word metric, a point $p\in \alpha$ is said to be {\it $(\epsilon,\eta)$-deep} if there is a $g \in \Gamma$, $P\in \Omega$ such that the part of $\alpha$ containing the points at distance at most $\eta$ from $p$
is contained in the $\epsilon$-neighborhood of $gP$. Otherwise, $p\in \alpha$ is called an \textit{$(\epsilon,\eta)$-transition point} of $\alpha$.
\end{definition}

By \cite[Proposition~2.4]{GePoJEMS}, any infinite word geodesic converges to a unique point in the Floyd boundary and any two distinct points in the Floyd compactification can be joined by a word geodesic. By Gerasimov's map, the same is true for the Bowditch compactification.

We also record several facts about word geodesics in relatively hyperbolic groups which will be used in the paper.

\begin{lemma}\label{transitionpoints}
For each $\epsilon>0$ there is an $\eta_{0}(\epsilon)>0$ such that the $\epsilon$-neighborhoods of any two distinct cosets of parabolic subgroups have intersection whose diameter is bounded above by $\eta_{0}(\epsilon)$.
Moreover, there is an $\epsilon_0>0$ such that the following holds for any $\epsilon>\epsilon_0$ and $\eta>\eta_{0}(\epsilon)$.
\begin{enumerate}[a)]
    \item \cite[Theorem 4.1] {DrutuSapir}, \cite[Proposition 5.6]{GePoCrelle} Let $\alpha$ be a geodesic and $v$ a point on $\alpha$.
    Assume that $v$ is $(\epsilon, \eta)$-deep in some coset $gP$ of a parabolic subgroup.
    Then the entry and exit point of $\alpha$ in $N_{\epsilon}(gP)$ are $(\epsilon,\eta)$-transition points.
    \item \cite[Lemma 2.20]{Yang} Any geodesic ray converging to a conical point on the Bowditch (or Floyd) boundary contains an unbounded sequence of $(\epsilon,\eta)$-transition points.
\end{enumerate}
\end{lemma}

The following proposition describes the connection between the Floyd metric and transition points.

\begin{proposition}\label{Floydgeo}
Let $\epsilon_0$ be as in Lemma \ref{transitionpoints}.
Then, the following holds.
\begin{enumerate}[a)]
\item \cite[Proposition 4.1]{GePoCrelle} For every $\epsilon\geq \epsilon_0$ and every $\delta>0$, there exists a number $D_{0}=D_{0}(\epsilon,\delta)>0$ such that for each coset $gH$ of a parabolic subgroup and each $w\in \Gamma$, any $z$ in the $\epsilon$-neighborhood $N_{\epsilon}(gH)$ of $gH$ satisfies either $d(z,w)\leq D_0$ or $\overline{\delta}^{f}_{w}(z,\xi)\leq \delta$ where $\xi\in \partial_B \Gamma$ is the unique parabolic limit point of $gH$.
\item \cite[Corollary 5.10]{GePoCrelle} For every $\eta>0$ and every $D>0$, there exists a number $\delta>0$ such that for $y\in \Gamma$ and $x,z\in \overline{\Gamma}_{B}$,
if $y$ is within word distance $D$ of an $(\epsilon,\eta)$-transition point of a word geodesic from $x$ to $z$ then $\overline{\delta}^{f}_{y}(x,z)>\delta$.
\item \cite[Lemma 2.10]{PotyagailoYang} For every $\epsilon>\epsilon_0$, every $\eta>\eta_0(\epsilon)$ and every $\delta>0$ there exist a number $D_1=D_1(\delta,\epsilon)>0$ such that for $y\in \Gamma$ and $x,z\in \overline{\Gamma}_{B}$, if $\overline{\delta}^{f}_{y}(x,z)>\delta$ then any word geodesic from $x$ to $z$ has an $(\epsilon,\eta)$-transition point within $D_1$ of $y$.
\end{enumerate}
\end{proposition}

We will use the following terminology.
Let $x,y,z$ be either points in the group $\Gamma$ or limit points in the Bowditch boundary $\partial_B\Gamma$.
A \emph{geodesic triangle} with vertices $x,y,z$ is the union of (finite or infinite) geodesics $[x,y]$, $[y,z]$, $[z,x]$ from $x$ to $y$, from $y$ to $z$ and from $z$ to $x$ respectively.
We will be mainly interested in geodesic triangles whose vertices are either in $\Gamma$ or are conical limit points.

\begin{lemma}\label{idealtriangles}
Fix $\epsilon\geq \epsilon_0$ and $\eta \geq \eta_0 (\epsilon)$.
There is a $D=D(\epsilon,\eta)>0$ such that for any geodesic triangle with vertices $x,y,z$, where $x,y,z$ are either elements of the groups or conical limit points, any $(\epsilon,\eta)$-transition point on one side is within $D$ of an $(\epsilon,\eta)$-transition point on one of the other two sides.
\end{lemma}

\begin{proof}
Suppose $v$ is an $(\epsilon,\eta)$-transition point on $[x,z]$.
Then, Proposition~\ref{Floydgeo}~b) shows that $\overline{\delta}^{f}_{v}(x,z)>\delta(\epsilon,\eta)>0$.
Suppose moreover that $v$ is more than $K$ away from $(x,y)$.
Then, by Karlsson's lemma \cite[Lemma~1]{Karlsson}, $\overline{\delta}^{f}_{v}(x,y)\leq \phi^{-1}(K)$ where $\phi^{-1}(K)\to 0$ as $K\to \infty$.
Thus, $\overline{\delta}^{f}_{v}(y,z)>\delta(\epsilon,\eta)-\phi^{-1}(K)>\delta(\epsilon,\eta)/2$ when $K$ is large enough.
It follows from Proposition~\ref{Floydgeo}~c) that $v$ is within $D=D(\epsilon, \delta(\epsilon,\eta)/2)$ of an $(\epsilon,\eta)$-transition point on $[y,z]$, completing the proof.
\end{proof}

\begin{remark}
This property is sometimes called the "relatively thin triangles property." It states that geodesic triangle are $D$-thin along transition points.
Note that for finite triangles (that is with vertices $g_1,g_2,g_3\in \Gamma$), this property is given by \cite[Proposition 4.6]{Sisto}.
It can also be derived using techniques of \cite[Section~8]{Hruska} or from the proof of \cite[Proposition~7.1.1]{GerasimovPotyagailo}.
\end{remark}

\subsection{Shadows and coverings}\label{Sectionshadowsandcoverings}
We fix a finite generating set for $\Gamma$ and denote by $d_w$ the corresponding word metric.
What we call geodesics are geodesic for this word distance.

Three kinds of shadows are defined by Yang in \cite{Yang}, what the author calls large, small and partial shadows.
We will be only interested in partial shadows in this paper, but other kinds of shadows can be useful in different contexts.

\begin{definition}
Let $\epsilon,\eta>0$ and let $r\geq 0$.
The partial shadow $\Omega_{r,\epsilon,\eta}(g)$ at $g \in \Gamma$ is the set of points $\xi\in \partial_B\Gamma$ such that there is a geodesic ray $[e,\xi)$ intersecting $B(g,r)$ and containing an $(\epsilon,\eta)$-transition point in $B(g,2\eta)$.
\end{definition}

We will write $\Omega_{\epsilon,\eta}(g)=\Omega_{2\eta,\epsilon,\eta}(g)$.
Note that $\Omega_{\epsilon,\eta}(g)$ simply consists of those points of $\partial_{B}\Gamma$ which can be connected to $e$ by a geodesic ray containing an $(\epsilon,\eta)$-transition point in $B_{2\eta}(g)$.
We will need the following application of a variant of Lebesgue's differentiation theorem, which will be proved in the next section.

\begin{proposition}\label{localsingularity}
Fix $\epsilon>\epsilon_0$.
Let $\kappa_1$ be a finite Borel measure on $\partial_B\Gamma$, giving full measure to conical limit points.
Assume that for a constant $C>0$, for all sufficiently large $\eta$,
$$\kappa_{1}(\Omega_{\epsilon,2\eta}(g))\leq C \kappa_{1}(\Omega_{\epsilon,\eta}(g)).$$
Let $\kappa_2$ be any finite Borel measure on $\partial_B\Gamma$.

Then the following holds for large enough $\eta>0$. 
\begin{enumerate}[a)]

   \item If $\kappa_1$ and $\kappa_2$ are mutually singular then for $\kappa_1$-almost every $\xi \in \partial_B\Gamma$ we have $$\lim_{t\to \infty} \sup_{g\in \Gamma: ||g||>t,\xi \in \Omega_{\epsilon,\eta}(g)}\kappa_{2}(\Omega_{\epsilon,\eta}(g))/\kappa_{1}(\Omega_{\epsilon,\eta}(g))= 0.$$
   
   \item If $\kappa_1$ and $\kappa_2$ are equivalent and have no atoms then for $\kappa_1$-almost every $\xi \in \partial_B\Gamma$ we have $$\lim_{t\to \infty} \sup_{g\in \Gamma: ||g||>t,\xi \in \Omega_{\epsilon,\eta}(g)}\log \kappa_{2}(\Omega_{\epsilon,\eta}(g))/\log \kappa_{1}(\Omega_{\epsilon,\eta}(g))= 1.$$
\end{enumerate}
\end{proposition}

\section{Lebesgue differentiation theorem}\label{SectionLebesguetheorem}
This section is devoted to the proof of Proposition~\ref{localsingularity}.
See Sections~2.8 and~2.9 of the book by Federer \cite{Federer} for background on Vitali relations and their applications to differentiation theorems. We were inspired by their use in \cite{MYJ} to study quasi-conformal measures for divergence type isometry groups of Gromov hyperbolic spaces.
Let $\Lambda$ be a metric space. A covering relation $\mathcal{C}$ is a subset of the set of all pairs $(\xi,S)$ such that $\xi \in S\subset \Lambda$. A covering relation $\mathcal{C}$ is said to be fine at $\xi \in \Lambda$ if there exists a sequence $S_n$ of subsets of $\Lambda$ with $(\xi,S_n)\in \mathcal{C}$ and such that the diameter of $S_n$ converges to $0$.

Let $\mathcal{C}$ be a covering relation.
For any measurable subset $E\subset \Lambda$, define $\mathcal{C}(E)$ to be the collection of subsets $S\subset \Lambda$ such that $(\xi,S)\in \mathcal{C}$ for some $\xi \in E$. A covering relation $\mathcal{C}$ is said to be a Vitali relation for a finite measure $\mu$ on $\Lambda$ if it is fine at every point of $\Lambda$ and if the following holds: 
if $\mathcal{C}'\subset \mathcal{C}$ is fine at every point of $\Lambda$ then for every measurable subset $E$, $\mathcal{C}'(E)$ has a countable disjoint subfamily $\{S_n\}$ such that $\mu(E\setminus \cup^\infty_{n=1}S_{n})=0$.

For a covering relation $\mathcal{V}$ and a real valued function $f$ on $\Lambda$ let
$(\mathcal{V})\lim_{S\to x}f$ denote
$$(\mathcal{V})\lim_{S\to x}f=\lim_{\epsilon \to 0} \sup \{f(z),(x,S)\in \mathcal{V}, z\in S, \diam(S)<\epsilon\}.$$
Similarly we define $(\mathcal{V})\limsup_{S\to x}f$ and $(\mathcal{V})\liminf_{S\to x}f$.
The following criterion guarantees a covering relation is Vitali, see \cite[Theorem 2.8.17]{Federer}.

\begin{proposition}\label{Vitalicriterion}
Let $\mathcal{V}$ be a covering relation on $\Lambda$ such that each
$S\in \mathcal{V}(\Lambda)$ is a closed bounded subset, $\mathcal{V}$ is fine at every point of $\Lambda$.
Let $\mu$ be a measure on $\Lambda$ such that for $\mu$ almost every $x$
$$\inf \{\diam(S):(x,S)\in V\}>0.$$
For a positive function $f$ on $\mathcal{V}(\Lambda)$, $S\in \mathcal{V}(\Lambda)$, and a constant $\tau>1$ define
$\tilde{S}$ to be the union of all $S'\in \mathcal{V}(\Lambda)$ which have nonempty intersection with $S$ and satisfy $f(S')\leq \tau f(S)$.

Suppose that for $\mu$-almost every $\xi\in \Lambda$ we have 
$$\limsup_{S\to \xi} f(S)+\mu(\tilde{S})/\mu(S)<\infty.$$ 
Then the relation $\mathcal{V}$ is Vitali for $\mu$.
\end{proposition}

\begin{remark}
To give a proper meaning to $f(S)+\mu(\tilde{S})/\mu(S)<\infty$, one has to ensure that $\mu(S)\neq 0$.
This is implicit in the book of Federer since the set of $\xi$ such that $(S\to \xi)$ in the above sense satisfying $\mu(S)=0$ has measure zero, see in particular \cite[Theorem~2.9.5]{Federer}.
Anyway, in the following, we will use partial shadows to define a Vitali relation, which have positive measure for the different measures we will use.
\end{remark}

Fix $\epsilon \geq \epsilon_0$. We will show that for large enough $\eta$, the following relation $\mathcal{V}_{\epsilon,\eta}$ is a Vitali relation for reasonable measures $\kappa_1$ on $\partial_B \Gamma$. 
For $\zeta\in \partial_{B}\Gamma$ parabolic, we declare $(\zeta,\{\zeta\})\in \mathcal{V}_{\epsilon,\eta}$.
For  $\zeta\in \partial_{B}\Gamma$ conical, we declare $(\zeta, \Omega_{\epsilon,\eta}(g))\in \mathcal{V}_{\epsilon,\eta}$ whenever $\zeta \in  \Omega_{\epsilon, \eta}(g)$.



\begin{proposition}
For large enough $\eta$ the above relation $\mathcal{V}_{\epsilon,\eta}$ is fine at every conical point $\xi \in \partial_B \Gamma$.
\end{proposition}

\begin{proof}
If $\zeta \in \partial_B \Gamma$ is parabolic, there is nothing to prove since $(\zeta,\{\zeta\})\in \mathcal{V}_{\epsilon,\eta}$.
Let $\zeta \in \partial_B \Gamma$ be conical.
Then, according to Lemma~\ref{transitionpoints}, for $\eta\geq \eta_{0}(\epsilon)$, a geodesic ray $\alpha$ connecting $e$ to $\zeta$ contains infinitely many $(\epsilon,\eta)$-transition points $g_n$.
By definition, the partial shadow $\Omega_{\epsilon,\eta}(g_n)$ contains $\zeta$.
For large enough $n$, it is contained in any neighborhood of $\zeta$, completing the proof.
\end{proof}

\begin{proposition}
Let $\mu$ be a Borel measure on $\partial_B \Gamma$ with the property that for all sufficiently large $\eta$, there exists a constant $C>0$ such that
$$\mu(\Omega_{\epsilon,2\eta}(g))\leq C \mu(\Omega_{\epsilon,\eta}(g)).$$
Also assume that $\mu$ gives full measure to the set of conical limit points.
Then for large enough $\eta$ the relation $\mathcal{V}_{\epsilon,\eta}$ is
a Vitali relation for $\mu$.
\end{proposition}

\begin{remark}
Notice that $C$ is allowed to depend on $\eta$.
\end{remark}

\begin{proof}
We verify that our relation satisfies the conditions of Proposition~\ref{Vitalicriterion}. 
First, let $(\xi,\Omega_{\epsilon,\eta}(g))\in \mathcal{V}_{\epsilon,\eta}$, where $\xi$ is conical, so that $\xi \in \Omega_{\epsilon,\eta}(g)$.
We define then $f(\xi,\Omega_{\epsilon,\eta}(g))=e^{-v\|g\|}$, where $v$ is the volume growth (defined in the introduction).
Let $(\xi,\{\xi\})\in \mathcal{V}_{\epsilon,\eta}$, where $\xi$ is parabolic.
We define then $f(\xi,\{\xi\})=1$.
In any case, set $\tau=e^{v\eta}>1.$

To apply Proposition~\ref{Vitalicriterion}, we only have to deal with $\mu$-almost every point, so that we can only consider conical limit points.
For $S=(\xi,\Omega_{\epsilon,\eta}(g))$ in $\mathcal{V}_{\epsilon,\eta}(\partial_B \Gamma)$ we want to consider the union $\tilde{S}$ of elements $S'$ with $(\xi',S')$ in $\mathcal{V}_{\epsilon,\eta}(\partial_{B}\Gamma)$ such that $S'$ intersects $S$ and $f(S')\leq \tau f(S)$.
This is contained in the union of parabolic limit points inside $\Omega_{\epsilon,\eta}(g)$ and the union of partial shadows $\Omega_{\epsilon,\eta}(h)$ intersecting $\Omega_{\epsilon,\eta}(g)$ with $\|h\|\geq \|g\|-\eta$.
We want to bound $f(S)+\mu(\tilde{S})/\mu(S)$.
Again, since $\mu$ is supported on the conical limit points, we only have to consider those $S'$ of the form $\Omega_{\epsilon,\eta}(h)$.

If $d_w(h,g)\leq 10\eta$ then 
 $\Omega_{\epsilon,\eta}(h) \subset \Omega_{\epsilon,6\eta}(g)$.
Suppose $d_w(h,g)> 10 \eta$ and let $\zeta \in \Omega_{\epsilon,\eta}(h) \cap \Omega_{\epsilon,\eta}(g)$.
Then some geodesic $[e,\zeta)$ passes first at an $(\epsilon,\eta)$-transition point in $B_{2\eta}(g)$ and then at one in $B_{2\eta}(h)$.  Let $\xi$ be any other point of $\Omega_{\epsilon,\eta}(h)$.
Let $p$ and $q$ be $(\epsilon,\eta)$-transition points on $[e,\zeta)$ and $[e,\xi)$ respectively within $2\eta$ of $h$.
By the relatively thin triangles property Lemma~\ref{idealtriangles} applied to the associated triangle $(e,p,q)$, the fact that $[e,\zeta)$ has an $(\epsilon,\eta)$-transition point in $B_{2\eta}(g)$ and the fact that $d_w(g,h)>10\eta$, we know that $[e,\xi)$ has a transition point in $B_{2\eta+D}(g)$. Thus, $\xi\in\Omega_{\epsilon,\eta+D}(g)$.
The number $D$ depends on $\eta$, but for large enough $\eta$ we have that $\tilde{S}\subset \Omega_{\epsilon,2\eta}(g)$.
By the assumption on the measure we have $\mu(\tilde{S})/\mu(S)$ bounded above, completing the proof.
\end{proof}

The following is contained in Theorems 2.9.5 and 2.9.7 of \cite{Federer}.
\begin{proposition}\label{RNexists}
Let $\Lambda$ be a metric space, $\kappa_1$ a finite Borel measure on $\Lambda$ and $\mathcal{V}$ a Vitali relation for $\kappa_1$.
Let $\kappa_2$ be any finite Borel measure on $\Lambda$. Define a new Borel measure $\psi(\kappa_{1},\kappa_{2})$ by 
$$\psi(\kappa_{1},\kappa_{2})(A)=\inf \{\kappa_2(B): B \textrm{ Borel },  \kappa_1(B\Delta A)=0\}.$$
This measure is absolutely continuous to $\kappa_1$.
The limit $$D(\kappa_1,\kappa_2,\mathcal{V},x)=(\mathcal{V})\lim_{S\to x}\kappa_2(S)/\kappa_1(S)$$ exists for $\kappa_1$-almost every $x$ and is equal to
$d\psi(\kappa_{1},\kappa_{2})/d\kappa_1$.
\end{proposition}

As a corollary we obtain:
\begin{corollary} \label{Vitalidichotomy}
\begin{enumerate}[a)]
    \item If the $\kappa_i$ are mutually singular, then $$(\mathcal{V})\lim_{S\to \zeta}\kappa_{2}(S)/\kappa_{1}(S)=0$$   for $\kappa_{1}$-almost every $\zeta \in \Lambda$.
    
    \item If the $\kappa_{i}$ are equivalent and non-atomic, then 
    $$(\mathcal{V})\lim_{S\to \zeta} \frac{\log \kappa_{2}(S)}{\log \kappa_{1}(S)}=1$$
    for $\kappa_{1}$-almost every $\zeta \in \Lambda$.
\end{enumerate}
\end{corollary}

\begin{proof}
If $\kappa_1$ and $\kappa_2$ are mutually singular, then by definition 
$\psi(\kappa_{1},\kappa_{2})=0$.
Together with Proposition~\ref{RNexists}, this proves a).

For b), assume $\kappa_1$ and $\kappa_2$ are equivalent. 
By Proposition \ref{RNexists} we have 
$$D(\kappa_1,\kappa_2,\mathcal{V},x)=(\mathcal{V})\lim_{S\to x}\kappa_{1}(S)/\kappa_{2}(S)>0$$
for $\kappa_1$ almost every $x$.
Thus, since $\kappa_1$ is non-atomic, as $(\mathcal{V})S\to x$, we have
\begin{equation*}
    \frac{\log \kappa_{2}(S)}{\log \kappa_{1}(S)}=1+\frac{\log(\kappa_{2}(S)/\kappa_{1}(S))}{\log (\kappa_{1}(S))}\to 1. \qedhere
\end{equation*}
\end{proof}

Proposition \ref{localsingularity} now follows by applying the Corollary \ref{Vitalidichotomy} to the Vitali relation defined by $\mathcal{V}_{\epsilon,\eta}$. \qed

\section{Rigidity of the Guivarc'h inequality}\label{Sectionrigidity}
The goal of this section is to prove Theorem~\ref{theoremBHM}.
We will follow the same strategy as in \cite{BlachereHassinskyMathieu2}.
We consider a non-elementary relatively hyperbolic group $\Gamma$.

\subsection{Quasi-conformality of the measures}\label{Sectionquasiconformality}
We fix a finite generating set $S$ for $\Gamma$ and denote by $d_w$ the corresponding word metric.
We consider geodesics for this word distance.

Let $\xi$ be a conical limit point.
Define the Busemann function $\beta_{\xi}(\cdot,\cdot)$ at $\xi$ as follows.
Let $\alpha$ be a geodesic ray in the Cayley graph $\Cay(\Gamma,S)$ ending at $\xi$.
Define then, for $g \in \Gamma$,
$b_{\alpha,\xi}(g)=\limsup [d_w(g,\alpha(t))-t]$.
The Busemann function $\beta_{\xi}$ is then defined as
$$\beta_{\xi}(g,g')=\sup(b_{\alpha,\xi}(g)-b_{\alpha,\xi}(g')),$$
where the supremum is taken over all geodesic rays $\alpha$ ending at $\xi$.
According to \cite[Lemma~2.20]{Yang}, there exists $C\geq 0$ such that for any $g,g'$, there exists a neighborhood $\mathcal{U}$ of $\xi$ in $\Gamma\cup \partial_B\Gamma$ such that
for all $g_0\in \mathcal{U}$,
$$\left |\beta_{\xi}(g,g')-[d_w(g,g_0)-d_w(g',g_0)]\right |\leq C.$$




Yang constructs in \cite{Yang} Patterson-Sullivan measures $\kappa_\gamma$ on $\partial_B\Gamma$ for every $\gamma\in \Gamma$ as follows.
Recall that we define the Poincar\'e series as
$$\mathcal{P}(s)=\sum_{g\in \Gamma}\mathrm{e}^{-s\|g\|}$$
where we recall that $\|g\|=d_w(e,g)$.
For $s>v$, the Poincar\'e series is convergent and for $s<v$, it is divergent.
The Patterson-Sullivan measure $\kappa_{g_0}$  at $g_0$ is then a weak limit point of
$$\frac{1}{\mathcal{P}(s)}\sum_{g\in \Gamma}\mathrm{e}^{-sd_w(g,g_0)}\delta_{g}$$
when $s$ tends to $v$ from above (where $\delta_{g}$ is the Dirac measure at $g$).

Yang proves that the Patterson-Sullivan measures $\kappa_{g}$ defines a $v$-dimensional quasi-conformal density in the following sense.

\begin{proposition}\cite[Lemma~4.14]{Yang}
For any $g$, $\kappa_{g}$ is without atoms, and in particular gives full weight to conical points.
Moreover, for any $g,g'$, the measures $\kappa_{g}$ and $\kappa_{g'}$ are equivalent and
$$\frac{d\kappa_{g}}{d\kappa_{g'}}(\xi)\asymp \mathrm{e}^{-v\beta_{\xi}(g,g')}$$
for $\kappa_{g}$-almost every conical limit point $\xi$.
\end{proposition}

Furthermore, Yang shows that the $\kappa_{g}$ are ergodic with respect to the $\Gamma$-action \cite[Appendix A]{Yang}.
Let $\kappa=\kappa_e$.
Yang proves an analogue of Sullivan's shadow lemma for these measures.

\begin{proposition}[Shadow Lemma for P.-S. Measure]\cite[Lemma~4.3, Lemma~5.2]{Yang}\label{Shadowkappa}
Let $\epsilon_0$ be as in Lemma \ref{transitionpoints}.
For each $\epsilon\geq \epsilon_0$ there is an $\eta(\epsilon)$ and an $r(\epsilon)>0$ such that for each $r\geq r(\epsilon)$, $\eta \geq \eta(\epsilon)$ and any $g \in \Gamma$ we have
$$\kappa(\Omega_{r,\epsilon,\eta}(g))\asymp \mathrm{e}^{-v\|g\|}$$
where the implicit constant depends on $r,\epsilon,\eta$ but not on $g$.
\end{proposition}

We are now interested in the same properties, but for the Green distance and the harmonic measure.
Let $\mu$ be an admissible measure on $\Gamma$ with finite super-exponential moment with respect to $d_w$.
Let $\nu_{g}$ be the harmonic measure of the random walk associated with $\mu$ starting at $g \in \Gamma$ and let $\nu=\nu_e$.
First, we define the Busemann function for the Green distance at a conical limit point.
According to \cite{GGPY}, conical limit points are in one-to-one correspondence with a subset of the Martin boundary.
Precisely, if $g_n$ is a sequence in $\Gamma$ converging to a conical limit point $\xi$, then
for every $g$, $\frac{G(g,g_n)}{G(e,g_n)}$ converges to some limit $K_{\xi}(g)$.
In particular, if $g,g'$ are fixed, then $\frac{G(g,g_n)}{G(g',g_n)}$ converges to $\frac{K_{\xi}(g)}{K_{\xi}(g')}$.
Define then
$$\beta^G_{\xi}(g,g')=-\log \left ( \frac{K_{\xi}(g)}{K_{\xi}(g')}\right ).$$
Then, 
$\beta^G_{\xi}(g,g')=\lim [d_G(g,g_n)-d_G(g',g_n)]$
for any sequence $g_n$ converging to $\xi$.


The harmonic measure $\nu$ is always conformal with respect to the Green metric in the sense that
$$\frac{d\nu_{g}}{d\nu}(\xi)=K_{\xi}(g)$$
for $\nu_{g}$-almost every conical limit point $\xi$,
see for example \cite[Theorem~24.10]{Woess}.
Before stating the shadow lemma for the harmonic measure, we prove the following lemma.

\begin{lemma}\label{partialshadowscontainballs}
For any $\epsilon \geq \epsilon_0$ there is an $\eta(\epsilon)>0$, an $r(\epsilon)>0$ and a $c>0$ such that for $r\geq r(\epsilon)$, $\eta \geq \eta(\epsilon)$ and all $g\in \Gamma$ we have that $g^{-1}\Omega_{r,\epsilon,\eta}(g)$ contains a ball of radius $c$ in the shortcut metric on $\partial_{B}\Gamma$.
\end{lemma}

\begin{proof}
Let $A>0$ be the diameter of $\partial_{B}\Gamma$ in the shortcut metric $\overline{\delta}^{f}_{e}$. Then the subset $S$ of $\partial_{B}\Gamma$ consisting of $\zeta$ such that $\overline{\delta}^{f}_{e}(g^{-1},\zeta)>A/10$ contains  a ball of radius $A/10$ in the restriction of this metric to $\partial_{B}\Gamma$. On the other hand, by Proposition \ref{Floydgeo}, $S$ is contained in $g^{-1}\Omega_{r,\epsilon,\eta}(g)$ for suitable $r,\eta>0$ depending only on $A$.
\end{proof}

We can now prove the following.

\begin{proposition}[Shadow Lemma for Harmonic Measure]\label{Shadownu}
For any $\epsilon\geq \epsilon_0$ there is an $\eta(\epsilon)>0$ and an $r(\epsilon)>0$ such that for each $r\geq r(\epsilon)$, $\eta \geq \eta(\epsilon)$ and any $g \in \Gamma$ we have 
$\nu(\Omega_{r,\epsilon,\eta}(g))\asymp \mathrm{e}^{-d_G(e,g)}$
where the implicit constant only depends on $r,\epsilon,\eta$ but not on $g$.
\end{proposition}
 
\begin{proof}
Let $\Omega(g)=\Omega_{r,\epsilon,\eta}(g)$.
The proof uses the relative Ancona inequalities~(\ref{relativeAncona}) proved in \cite{GGPY}.
We have
$$\nu(\Omega(g))=\int_{\xi \in \Omega(g)}\frac{1}{K_{\xi}(g)}d\nu_{g}(\xi)=\int_{\xi \in \Omega(g)}\mathrm{e}^{-\beta^{G}_{\xi}(e,g)}d\nu_{g}(\xi).$$
By definition of the relative shadow, for each $\xi\in \Omega(g)$ there exists a sequence $g_n$ converging to $\xi$ such that $g$ is at most $r$ away from an $(\epsilon,\eta)$-transition point on a geodesic from $e$ to $g_n$.
The relative Ancona inequalities~(\ref{relativeAncona})
show that
$$G(e,g_n)\asymp G(e,g)G(g,g_n),$$
where the implicit constant only depends on $r,\eta,\epsilon$.
Thus,
$K_{\xi}(g)\asymp \frac{1}{G(e,g)}$.
We can reformulate this as
$$d_{G}(e,g)-C\leq \beta^{G}_{\xi}(e,g)\leq d_{G}(e,g)+C$$
where $C$ only depends on $r,\eta,\epsilon$.
Furthermore,
$\nu_{g}(\Omega(g))=\nu(g^{-1}\Omega(g))$.
According to Lemma~\ref{partialshadowscontainballs}, the partial shadow $g^{-1}\Omega(g)$ always contains a ball for the shortcut metric of uniform radius independent of $g$, and since by \cite[Theorem 9.4]{GGPY} $\nu$ gives full measure to $\partial_{B}\Gamma$, we have
$$\nu(g^{-1}\Omega(g))\geq c$$
for some uniform $c>0$.
The result now follows. 
\end{proof}
Proposition \ref{Shadownu} implies the doubling property
$$\nu(\Omega_{\epsilon,2\eta}(g))\leq C_{\epsilon,\eta} \nu(\Omega_{\epsilon,\eta}(g))$$
and allows us to apply Proposition \ref{localsingularity} to the measure $\nu$ in place of $\kappa_1$.

\subsection{Deviation inequalities for relatively hyperbolic groups}\label{Sectiondeviationinequalities}
As before, we consider a word metric on a non-elementary relatively hyperbolic group $\Gamma$ and a probability measure $\mu$ with finite super-exponential moment.
For a sample path $\omega$ let $\omega_n$ be its $n$-th position and $\omega_\infty \in \partial_B \Gamma$ its limit point, which exists almost surely. Recall that the harmonic measure $\nu$ satisfies $\nu(A)=P(\omega_\infty \in A)$ for any Borel set $A\subset \partial_B \Gamma$.

The following results control the deviation of sample paths of the random walk from geodesics in the word metric.
\begin{lemma}\label{closetogeodesics} \cite[Theorem 10.6]{Mathieu-Sisto}
There exists $C>0$ such that for each $k$, $n\geq k$ and for each $a>1$, we have
$$P(\sup_{\alpha}\ d_w(\omega_k, \alpha)>a)<Ce^{-a/C},$$
where the supremum is taken over all geodesics $\alpha$ for the word metric from $e$ to $\omega_n$.
\end{lemma}
For a finite or infinite geodesic $\alpha$ let $\Tr_{\epsilon,\eta} \alpha$ denote the set of $(\epsilon,\eta)$-transition points on $\alpha$.
\begin{lemma}\label{closetotransitionpoints}
For every $\epsilon \geq \epsilon_0,\eta>\eta_{0}(\epsilon)$, there exists a function $F_{\epsilon,\eta}:\R\to \R$ satisfying that $F_{\epsilon,\eta}(t)\to 0$ as $t\to \infty$ such that for each $k$ we have
$$P(\sup_{\alpha}\ d_w(\omega_k,\Tr_{\epsilon,\eta} \alpha)>a)<F_{\epsilon,\eta}(a),$$
where the supremum is taken over all geodesic rays $\alpha$ for the word metric from $e$ to $\omega_{\infty}$.
\end{lemma}

\begin{proof}
By Proposition \ref{Floydgeo}~c) there is a function $W(t)$ which decays to zero as $t\to \infty$ such that for any $g,h\in \Gamma$ and $\xi\in \partial_B\Gamma$
$d(g,\Tr_{\epsilon,\eta}[h,\xi))>D$ implies
$\overline{\delta}^{f}_{g}(h,\xi)<W(D)$.
Thus, it suffices to show that
$\sup_{k\geq 1}P(\overline{\delta}^{f}_{\omega_k}(e,\omega_\infty)<t)\to 0$,
as $t\to 0$.
We have
\begin{align*}
    P(\overline{\delta}^{f}_{\omega_k}(e,\omega_\infty)<t)&=\sum_{g\in \Gamma}
P(\overline{\delta}^{f}_{\omega_k}(e,\omega_\infty)<t,\omega_k=g)\\
&=\sum_{g\in \Gamma}
P(\overline{\delta}^{f}_{g}(e,g\omega^{-1}_k \omega_\infty)<t,\omega_k=g).
\end{align*}
The random variables $\omega_k$ and $\omega^{-1}_k\omega_{\infty}$ are independent, and  $\omega^{-1}_k\omega_{\infty}$ has the same distribution as $\omega_\infty$, so that we get
\begin{align*}
P(\overline{\delta}^{f}_{\omega_k}(e,\omega_\infty)<t)&=\sum_{g\in \Gamma}P(\overline{\delta}^{f}_{g}(e,g \omega_\infty)<t) P(\omega_k=g)\\
&=\sum_{g\in \Gamma}P(\overline{\delta}^{f}_{e}(g^{-1},\omega_\infty)<t) P(\omega_k=g)\\
&=\sum_{g\in \Gamma}\nu\left (\left \{\xi\in \partial_B\Gamma, \overline{\delta}^{f}_{e}(g^{-1},\xi)<t\right \}\right )P(\omega_k=g).
\end{align*}
The last following from the fact that $\nu$ is the law of $\omega_{\infty}$.
Letting
$$Q(t)=\sup_{g\in \Gamma} \nu\left (\left \{\xi\in \partial_B\Gamma, \overline{\delta}^{f}_{e}(g^{-1},\xi)<t\right \}\right ),$$
we obtain 
$$P(\overline{\delta}^{f}_{\omega_k}(e,\omega_\infty)<t)\leq Q(t).$$
It thus suffices to prove that $Q(t)\to 0$ as $t\to 0$.
If not, there is a $c>0$ and a sequence $z_n\in \Gamma$ with 
$$\nu\left (\left \{\xi\in \partial_B\Gamma, \overline{\delta}^{f}_{e}(z_n,\xi)<1/n\right \}\right )>c.$$
Up to taking a subsequence, we can assume that $z_n\to \alpha \in \partial_B\Gamma$ so that we have $\overline{\delta}^{f}_{e}(w_n,\alpha)<1/n$.
Then by the triangle inequality for $\overline{\delta}^{f}_{e}$ we have
$$\nu\left (\left \{\xi\in \partial_B\Gamma, \overline{\delta}^{f}_{e}(\xi,\alpha)<1/n\right \}\right )>c$$
for each $n$.
This implies $\nu(\{\alpha\})>c$ which is impossible since $\nu$ has no atoms.
\end{proof}

\subsection{Guivarc'h inequality and comparison of the measures}\label{Section1implies2}
The goal of this subsection is to prove the following step in the proof of Theorem~\ref{theoremBHM}.
We consider a word metric on $\Gamma$ and a probability measure $\mu$ with finite super-exponential moment.

\begin{proposition}\label{(1)implies(2)BHM}
The harmonic measure and the Patterson-Sullivan measure are equivalent if and only if $h=lv$.
\end{proposition}
Recall that we use the notation $\Omega_{\epsilon,\eta}(g)=\Omega_{2\eta,\epsilon,\eta}(g)$ to denote the set of $\xi \in \partial_B \Gamma$ such that some geodesic $[e,\xi)$ has an $(\epsilon,\eta)$-transition point within $2\eta$ of $g$.
For the rest of the paper we fix the following convention on $\epsilon,\eta$.
We choose an $\epsilon>\epsilon_0$ and an $\eta=\eta(\epsilon)$ large enough to satisfy Lemma~\ref{transitionpoints} and Proposition~\ref{Floydgeo}.
Moreover, $\eta$ is chosen large enough so that the shadow lemmas (Propositions~\ref{Shadowkappa} and~\ref{Shadownu}) hold for $r=2\eta$, and Proposition~\ref{localsingularity} also holds.

Define for a sample path $\omega$
$$\phi_n=\phi_n(\omega)=\frac{\kappa(\Omega_{\epsilon,\eta}(\omega_n))}{\nu(\Omega_{\epsilon,\eta}(\omega_n))}.$$
Let $\psi_n=\ln \phi_n$.
For simplicity, we do not refer to $\eta$ and $\epsilon$ in the notation for $\phi_n$ and $\psi_n$.
We will also write $\Omega(g)=\Omega_{\epsilon,\eta}(g)$ when not ambiguous.
Notice that the expectation of $\phi_n$ is given by
$$E(\phi_n)=\sum_{g\in \Gamma}\mu^{*n}(g)\frac{\kappa(\Omega(g))}{\nu(\Omega(g))}.$$

\begin{proposition}\label{lemma4.14Haissinsky}
There exists $C_1>0$ such that for any $N\geq 1$ we have 
$$\frac{1}{N}\sum^{N}_{n=1}E(\phi_n)\leq C_1.$$
\end{proposition}

\begin{proof}
Recall that the notation $f\lesssim g$ means that $f\leq C g$ for some constant $C$.
Consider $n,N$ with $1\leq n\leq N$. We will first show that there is some $k>0$ such that the quantity
$$R_{k}=\sum_{g\in \Gamma: \|g\|\geq k N} \frac{\kappa(\Omega(g))}{\nu(\Omega(g))}\mu^{*n}(g)$$ is bounded independently of $n,N$.

Indeed, by the shadow lemma for harmonic measure
\begin{equation}\label{equationnu}
   \nu(\Omega(g))\asymp G(e,g)=\sum^{\infty}_{n=1}\mu^{*n}(g).
\end{equation}
By the shadow lemma for the Patterson-Sullivan measure,
\begin{equation}\label{equationkappa}
    \kappa(\Omega(g))\asymp \mathrm{e}^{-v\|g\|}.
\end{equation}
Furthermore, since $d_{w}$ and $d_{G}$ are quasi-isometric (see for example \cite[Lemma~4.2]{Haissinsky} and note that the proof does not use the symmetry assumption),
$$\frac{\kappa(\Omega(g))}{\nu(\Omega(g))}\lesssim e^{c\|g\|}$$ for a constant $c$. 
This yields
$$R_{k}\lesssim \sum_{m\geq kN} e^{cm} \sum_{g: \|g\|\in [m,m+1)}\mu^{*n}(g)\leq \sum_{m\geq kN} e^{cm}P(\|\omega_n\|\geq m).$$

Since $\mu$ has finite super-exponential moment, we can apply the exponential Chebyshev inequality with exponent $2c$ to obtain
$$R_{k}\lesssim \sum_{m\geq kN} e^{-cm}E(e^{2c\|\omega_n\|}).$$
Since $n\leq N$ we have $$\|\omega_n\|\leq \sum^{N-1}_{j=0}\|\omega^{-1}_{j}\omega_{j+1}\|$$
from which we obtain, since the $\omega^{-1}_{j}\omega_{j+1}$ are independent random variables,
$$E(e^{2c\|\omega_n\|})\leq E_0^{N}$$ where $E_0=\sum_{g\in \Gamma}e^{2c\|g\|}\mu(g)$.
Choosing $k\geq \frac{2}{c}\log E_0$ we thus obtain 
$$R_{k}\lesssim \sum_{m\geq kN} e^{-cm}E(e^{2c\|\omega_n\|})\lesssim e^{-ckN}E_0^{N}\lesssim 1$$ giving us the desired estimate for $R_{k}$.

Now, we will show that the quantity
$$P_{N}=\frac{1}{N}\sum^{N}_{n=1}\sum_{g\in \Gamma:\|g\|\leq kN}\frac{\kappa(\Omega(g))}{\nu(\Omega(g))}\mu^{*n}(g)$$ is bounded independently of $N$.
Together with the estimate on $R_k$ this will prove the proposition.
Interchanging the order of summation we get, using~(\ref{equationnu}),
$$P_{N}=\frac{1}{N}\sum_{g\in \Gamma:\|g\|\leq kN}\frac{\sum^{N}_{n=1}\mu^{*n}(g)}{\nu(\Omega(g))}\kappa(\Omega(g))
\lesssim \frac{1}{N}\sum_{g:\|g\|\leq kN}\kappa(\Omega(g)).$$
By \cite[Theorem 1.9]{Yang}, for large enough $a>0$, we have
$$|\{g\in \Gamma: n-a<\|g\|\leq n\}|\asymp e^{vn}.$$
Consequently,
\begin{align*}
    \sum_{g \in \Gamma:1\leq \|g\|\leq kN}\kappa(\Omega(g))&\lesssim \sum^{N}_{n=1}\sum_{g:k(n-1)< \|g\|\leq kn} e^{-vkn}\\
    & \lesssim \sum^{N}_{n=1} e^{vkn}e^{-vkn}\lesssim N.
\end{align*}
The estimate for $P_N$ follows.
\end{proof}

\begin{proposition}\label{lemma4.15Haissinskya}
We have that $\psi_{n}/n$ converges to $h-lv$ almost surely and in expectation. 
\end{proposition}

\begin{proof}
By the shadow lemmas (Propositions~\ref{Shadowkappa} and~\ref{Shadownu}), we have
$$\frac{\psi_n}{n}=\frac{d_{G}(e,\omega_n)}{n}-\frac{v\|\omega_n\|}{n}+O\left (\frac{1}{n}\right ).$$
According to \cite[Theorem~1.1]{BlachereHassinskyMathieu1}, $\frac{1}{n}d_G(e,\omega_n)$ almost surely converges to $h$ whenever $\mu$ has finite entropy $H(\mu)$, which is implied by finite first moment, so in particular by finite super-exponential moment.
In other words, entropy is equal to the drift of $d_G$.
Thus, Kingman's ergodic theorem \cite[Theorem~10.1]{Walters} implies that $\psi_{n}/n$ converges to $h-lv$ almost surely and in expectation.
\end{proof}

\begin{proposition}\label{lemma4.15Haissinskyb}
There exists $C_2>0$ such that the sequence $E(\psi_n)+C_2$ is sub-additive.
\end{proposition}

\begin{proof}
Let $m,n\geq 1$.
The shadow lemma for the Patterson-Sullivan measure and the triangle inequality for $d_G$ implies that
\begin{align*}
E(\psi_{n+m})-(E(\psi_{n})+E(\psi_{m}))&\lesssim v E(\|\omega_n\|+ \|\omega_m\|-\|\omega_{n+m}\|)\\
&\leq 2vE(d_w(\omega_n,[e,\omega_{n+m}])).
\end{align*}
By Lemma~\ref{closetogeodesics}, the last expression is bounded by a constant $C_2$, so that $E(\psi_n)+C_2$ is sub-additive.
\end{proof}

\begin{proposition}\label{lemma4.16Haissinsky}
We have the following two cases.
\begin{enumerate}[a)]
    \item If $\kappa$ and $\nu$ are not equivalent, then $\phi_n$ converges to 0 in probability.
    \item If $\kappa$ and $\nu$ are equivalent then $\frac{\log \kappa(\Omega_{\epsilon,\eta}(\omega_n))}{\log \nu(\Omega_{\epsilon,\eta}(\omega_n))}$ tends to 1 in probability.
\end{enumerate}
\end{proposition}

\begin{remark}
For a), note that in \cite{BlachereHassinskyMathieu2}, the analogous result for hyperbolic groups
states that $\phi_n$ converges to 0 almost surely.
This is actually a mistake, which is corrected in \cite[Lemma~4.16]{Haissinsky}.
\end{remark}

\begin{proof}
\begin{enumerate}[a)]
\item \noindent
First, both measures are ergodic with respect to the action of $\Gamma$ on $\partial_B\Gamma$.
Thus, if they are not equivalent, they are mutually singular.
Let $\alpha>0,c>0$.
By Lemma \ref{closetotransitionpoints} we have $P(\omega_\infty \notin \Omega_{\epsilon,D}(\omega_k))\leq F(D)$ independently of $k$ where $F(D)\to 0$ as $D\to \infty$.
Fix $D$ so that
$$P(\omega_\infty \notin \Omega_{\epsilon,D}(\omega_k))\leq \alpha$$
for all $k$.
By Proposition~\ref{localsingularity}~a) we have, for $\nu$-almost every $\xi$,
$$\lim_{t\to \infty}\sup_{||g||>t,\xi\in \Omega_{\epsilon,D}(g)}\frac{\kappa(\Omega_{\epsilon,D}(g))}{\nu(\Omega_{\epsilon,D}(g))}\to 0.$$
The shadow lemma for $\nu$ shows that $\nu(\Omega_{\epsilon,D}(g))\leq C \nu(\Omega_{\epsilon,\eta}(g))$ where $C$ depends only on $\epsilon,\eta,D$.
Thus the quantity 
$$R_{t}(\xi)=\sup_{||g||\geq t,\xi\in \Omega_{\epsilon,D}(g)}\frac{\kappa(\Omega_{\epsilon,\eta}(g))}{\nu(\Omega_{\epsilon,\eta}(g))}$$  converges to $0$ as $t\to \infty$ for $\nu$-almost every $\xi$.
Furthermore, for almost every sample path $\omega$, we have 
$||\omega_n||\to \infty$. 
Thus, by Egorov's theorem, we may choose a subset $E\subset \Gamma^{\mathbb{N}}$ of sample paths with $P(E^c)<\alpha$ and such that $||\omega_n||\to \infty$ and $R_{t}(\omega_\infty)\to 0$ uniformly over $\omega\in E$.
It follows that $R_{||\omega_n||}(\omega_\infty)\to 0$ uniformly over $\omega\in E$. This means that for large enough $n$, the conditions $\omega\in E$ and $\frac{\kappa(\Omega_{\epsilon,\eta}(\omega_n))}{\nu(\Omega_{\epsilon,\eta}(\omega_n))}\geq c$ imply $\omega_\infty \notin \Omega_{\epsilon,D}(\omega_n)$. The latter has probability at most $\alpha$, so we get
$$P(\phi_n\geq c)\leq P(E^c)+P(\omega_\infty \notin \Omega_{\epsilon,D}(\omega_n))\leq 2\alpha.$$
As $\alpha,c>0$ were chosen arbitrarily we get $P(\phi_n\geq c)\to 0$ as $n\to \infty$ for each $c>0$ so $\phi_n \to 0$ in probability.

\item \noindent This time, we define for each $D>0$
$$R_{t}(\xi)=
\sup_{||g||\geq t,\xi\in \Omega_{\epsilon,D}(g)}\left |\frac{\log \kappa(\Omega_{\epsilon,\eta}(g))}{\log \nu(\Omega_{\epsilon,\eta}(g))}-1\right|.$$
Using b) of Proposition~\ref{localsingularity} we obtain that for every $D$, 
$$\lim_{t\to \infty}R_{t}(\xi)= 1$$
for $\nu$-almost every $\xi$.
The proof is then similar to a). \qedhere
\end{enumerate}
\end{proof}

We now prove Proposition~\ref{(1)implies(2)BHM}.
\begin{proof}
Assume that the Patterson-Sullivan and the harmonic measures are not equivalent.
Let $\beta>0$.
Let $A_n$ be the event $\{\phi_n\geq \beta\}$ and $B_n=A_n^c$.
For every $n$,
$$E(\psi_n)=\int_{A_n}\psi_ndP+\int_{B_n}\psi_ndP.$$
According to Proposition~\ref{lemma4.16Haissinsky}, $\phi_n$ converges to 0 in probability. Thus, there exists $n_0$ such that for every $n\geq n_0$, $P(B_n)\geq 1-\beta$. In particular,
$$\int_{B_n}\psi_ndP\leq P(B_n)\log \beta \leq (1-\beta)\log \beta.$$
Let $C_1$ be the constant in Proposition~\ref{lemma4.14Haissinsky}.
Jensen inequality shows that $$\int_{A_n}\psi_ndP=P(A_n)\int_{A_n}\log \phi_n \frac{dP}{P(A_n)}\leq P(A_n)\log \left (\int_{A_n}\phi_n \frac{dP}{P(A_n)}\right ).$$
Rewrite the right-hand side as
$$P(A_n)\log \left (\int_{A_n}\phi_n \frac{dP}{P(A_n)}\right )=-P(A_n)\log P(A_n)+P(A_n)\log \left (\int_{A_n}\phi_n dP\right ).$$
The function $x\mapsto x\log x$ is first decreasing then increasing on $[0,1]$, so if $\beta$ is small enough,
$-P(A_n)\log P(A_n)\leq -\beta \log \beta$.
Moreover,
$$P(A_n)\log \left (\int_{A_n}\phi_n dP\right )\leq \beta \sup \left (0,\log E(\phi_n)\right ).$$
We thus have $$\int_{A_n}\psi_ndP\leq \beta\log \frac{1}{\beta}+\beta \sup \left (0,\log E(\phi_n)\right ).$$
According to Proposition~\ref{lemma4.14Haissinsky}, $\liminf E(\phi_n)\leq 2C_1$, so that there exists $p\geq n_0$ such that
$E(\phi_p)\leq 2C_1$.
In particular, for every small enough $\beta$, we can find $p$ such that
$$E(\psi_p)\leq (1-\beta)\log \beta+\beta\log \frac{1}{\beta}+\beta|\log 2C_1|.$$
The right-hand side tends to $-\infty$ when $\beta$ goes to 0.
If $\beta$ is small enough, we thus have for some $p$
$$E(\psi_p)+C_2\leq -1,$$
where $C_2$ is the constant in Proposition~\ref{lemma4.15Haissinskyb}.
Since $E(\psi_p)+C_2$ is sub-additive, we have
$$\frac{1}{k}(E(\psi_{kp})+C_2)\leq E(\psi_p)+C_2\leq -1.$$
Finally, $\frac{1}{kp}E(\psi_{kp})$ converges to $h-lv$ by Proposition~\ref{lemma4.15Haissinskya}, so letting $k$ tend to infinity, we get
$$h-lv\leq -\frac{1}{p}<0.$$
Thus, $h<lv$.


Conversely, suppose the measures are equivalent.
By the shadow lemma for the Patterson-Sullivan measure $\kappa$ we have
$$\frac{-\log \kappa(\Omega(g))}{\|g\|}\to v$$
as $\|g\|\to \infty$ and in particular 
$$\frac{-\log \kappa(\Omega(\omega_n))}{\|\omega_n\|}\to v$$
for almost every sample path.

Furthermore, for almost every sample path we have
$$\lim_{n\to \infty}\frac{-\log \nu(\Omega(\omega_n))}{\|\omega_n\|}=\lim_{n\to \infty}\frac{d_{G}(e,\omega_n)}{\|\omega_n\|}=\lim_{n\to \infty}\frac{d_{G}(e,\omega_n)/n}{\|\omega_n\|/n}=\frac{h}{l}.$$
Thus, almost surely, $$\frac{\log \kappa(\Omega(\omega_n))}{\log \nu(\Omega(\omega_n))}\to \frac{vl}{h}.$$
As the measures are equivalent we have according to Proposition~\ref{lemma4.16Haissinsky},
$$\frac{\log \kappa(\Omega_{\epsilon,\eta}(\omega_n))}{\log \nu(\Omega_{\epsilon,\eta}(\omega_n))}\to 1$$ in probability, which ensures that $h=lv$.
\end{proof}

The following result is a consequence of Proposition~\ref{(1)implies(2)BHM}.
Indeed, notice that $h$ and $l$ are the same for the measure $\mu$ and the reflected measure $\check{\mu}$.
Actually, $H(\mu)=H(\check{\mu})$ and $L(\mu)=L(\check{\mu})$.
\begin{corollary}\label{reflectedequivalence}
Let $\check{\nu}$ be the harmonic measure for the reflected random walk $\check{\mu}$. Then $\check{\nu}$ is equivalent to $\kappa$ whenever $\nu$ is equivalent to $\kappa$.
\end{corollary}

\subsection{Radon-Nikodym derivatives}\label{Section2implies3}
In this subsection, we prove the following result.

\begin{proposition}\label{(2)implies(3)BHM}
If the harmonic measure $\nu$ and its reflection $\check{\nu}$ are both equivalent to the Patterson-Sullivan measure $\kappa$, then the Radon-Nikodym derivative $d\kappa/d\nu$ is bounded away from 0 and infinity.
\end{proposition}
This will use the following general lemma.

\begin{lemma}\label{lemmaergodicdouble}
Let $Z$ be a compact metrizable space and let $G$ act by homeomorphisms. 
Let $\nu_1,\nu_2,\kappa_1,\kappa_2$ be Borel probability measures with full support on $Z$ and with $\nu_i$ equivalent to $\kappa_i$ for $i=1,2$.

Assume $G$ preserves the measure class of $\nu_i$ and acts ergodically on $(Z\times Z, \nu_1 \otimes \nu_2)$ and $(Z\times Z, \kappa_1 \otimes \kappa_2)$ for $i=1,2$.
Suppose there are positive bounded away from 0 measurable functions $f_\nu,f_\kappa:Z\times Z\setminus \Diag \to \mathbb{R}$, bounded on compact subsets of $Z\times Z \setminus \Diag$ such that
$m_\nu=f_\nu \nu_1 \otimes \nu_2$ and $m_\kappa=f_\kappa \kappa_1 \otimes \kappa_2$ are $G$-invariant ergodic Radon measure on 
$Z\times Z\setminus \Diag $.
Then for each $i=1,2$, $d\nu_i/d\kappa_i$ is bounded away from $0$ and $\infty$.
\end{lemma}

\begin{proof}
Since $\nu_i$ and $\kappa_i$ are equivalent, we have $d\nu_i=J_i d\kappa_i$ for a measurable positive function $J_i$. We want to show $J_i$ is $\kappa_i$-essentially bounded. This will prove that $d\nu_i/d\kappa_i$ is bounded away from $\infty$.
Interchanging the role of $\nu_i$ and $\kappa_i$, we will then get that it is also bounded away from 0.

Since $m_\nu$ and $m_\kappa$ are $G$-invariant ergodic measures, either they are singular or they are scalar multiples of each other. Thus, the assumption $d\nu_i=J_i d\kappa_i$ implies they are scalar multiples of each other.
Without loss of generality, we can assume that they coincide.
Note that we have
$$dm_{\nu}(a,b)=J_{1}(a)J_{2}(b)f_{\nu}(a,b)d\kappa_1(a)d\kappa_2(b)$$
so we have
$J_{1}(a)J_{2}(b)=f_{\kappa}(a,b)/f_{\nu}(a,b)$ for  $\nu_1\otimes \nu_2$-almost all $(a,b)$. 

Let $U,V$ be disjoint closed subsets in $Z$ with nonempty interior.
There is a $p\in V$ such that $J_{1}(a)J_{2}(p)=f_{\kappa}(a,p)/f_{\nu}(a,p)$ for $\nu_1$-almost all $a\in U$.
Dividing and noting that the $f_\nu$ and $f_\kappa$  are positive and bounded away from 0 and infinity on $U\times V$, we see that $C^{-1}_U<J_{1}(a)/J_{1}(a')<C_U$ for $\nu_1$-almost all $a,a'\in U$. Thus, $J_1$ is $\nu_1$-essentially bounded on any subset $U$ whose complement has nonempty interior. Covering $Z$ by two such sets, we see that $J_1$ is essentially bounded. The same argument applies to $J_2$
\end{proof}

To complete the proof of Proposition \ref{(2)implies(3)BHM} we just need to show that $\kappa \otimes \kappa$ and $\nu \otimes \check{\nu}$ can both be scaled by functions $f_\kappa$ and $f_\nu$ to obtain $\Gamma$-invariant Radon measures on $\partial^{2}_B\Gamma =\partial_B\Gamma\times \partial_B\Gamma \setminus \Diag$. 

For the harmonic measure, we may take $f_\nu$ to be the Naim kernel
defined for distinct conical points $\zeta,\xi$ as
$$\Theta(\zeta,\xi)=\liminf_{g\in \Gamma \to \zeta} \lim_{h\in \Gamma \to \xi}\frac{G(g,h)}{G(g,e)G(e,h)}=\liminf_{g\to \zeta}\frac{K_{\xi}(g)}{G(g,e)}.$$
The construction is done in \cite[Corollary~10.3]{GGPY}.

For the Patterson-Sullivan measure $\kappa$, the construction is done in the proof of \cite[Proposition~9.17]{GTT}, but we review it for completeness.
Let
$$\rho_{e}(\zeta ,\xi)= \limsup_{z\to\zeta, x\to \xi} 
\left( \frac{d_w(e,x)+d_w(e,z)-d_w(x,z)}{2} \right).$$
Define a measure $m'$ on $(\partial_B\Gamma \times \partial_B\Gamma)\setminus \Diag$ by 
$$dm'(\zeta, \xi)=e^{2v \rho_{e}(\zeta,\xi)}\ d\kappa(\zeta)\ d\kappa(\xi).$$

We first need to show that $m'$ is locally finite on $\partial_{B}\Gamma$. This amounts to showing that $\rho_{e}(\zeta ,\xi)$ is bounded outside of any neighborhood of the diagonal. Indeed, suppose $\zeta,\xi$ satisfy $\overline{\delta}^{f}_{e}(\zeta,\xi)\geq c>0$ and suppose $x_{n},y_{n}$ are sequences in $\Gamma$ converging to $\zeta$ and $\xi$ respectively. For large enough $n$ we have $\overline{\delta}^{f}_{e}(\zeta,\xi)\geq c/2>0$. By Proposition \ref{Floydgeo}, this implies any geodesic from $x_n$ to $y_n$ passes within $D$ of $e$, where $D$ depends only on $c$. 
This implies $d_{w}(x_n,e)+d_{w}(e,y_n)-d_{w}(x_n,y_n)<2D$. 
Taking limits, we see that $\rho_{e}(\zeta ,\xi)<2D$ completing the proof that $m'$ is locally finite.
Now we show that $m'$ is $\Gamma$-quasi-invariant with a uniformly bounded derivative. 
Indeed, we can compute
\begin{align*}
&2 \rho_e(g^{-1} \zeta, g^{-1}\xi) - 2 \rho_e(\zeta, \xi) \\
&= \limsup_{z \to \zeta, x \to \xi} \left[ d(e, g^{-1}x) + d(e, g^{-1}z) - d(g^{-1}x, g^{-1}z) \right]\\
&\hspace{5cm}- \limsup_{z \to \zeta, x \to \xi} \left[ d(e, x) + d(e, z) - d(x, z) \right]\\
&= \limsup_{x \to \xi} \left[ d(g, x) - d(e, x) \right] + \limsup_{z \to \zeta} \left[ d(g, z) - d(e, z) \right] + O(1)\\
&= \beta_\xi(g, e) + \beta_\zeta(g, e) + O(1).
\end{align*}
Note that we could distribute the limsup since the limsup and liminf are within bounded difference by \cite[Lemma 2.20] {Yang}.
Hence, combining this with the quasi-conformality of $\kappa$ one gets that the Radon-Nykodym cocycle of $m'$ is uniformly bounded for the $\Gamma$ action, that is,
$$\frac{dg_\star m'}{dm'} (\zeta, \xi)=e^{2v \rho_{e}(g^{-1} \zeta, g^{-1} \xi)- 2v \rho_{e}(\zeta, \xi)}\ \frac{dg_\star\kappa}{d\kappa}(\zeta)\ \frac{dg_\star\kappa}{d\kappa}(\xi) \asymp 1.$$
Hence, by a general fact in ergodic theory the Radon-Nykodym cocycle is also a coboundary (see \cite{Furman}, Proposition 1). Thus, there exists a $\Gamma$-invariant measure $m_{\kappa}$ on $(\partial_B\Gamma \times \partial_B\Gamma)\setminus \Diag$ in the same measure class as $m'$. In other words, one can take $f_\kappa$ to be within a bounded multiplicative constant of $e^{2v \rho_{e}(\zeta,\xi)}$. 
The functions $f_{\kappa}$ and $f_{\nu}$ thus constructed are bounded on compact subsets of $\partial^{2}_B\Gamma$.

The $\Gamma$-action on the square of the Poisson boundary
is ergodic (see \cite[Theorem~6.3]{Kaimanovich}) and since $\nu$ and $\check{\nu}$ both are equivalent to $\kappa$, the $\Gamma$-action on $\partial_B\Gamma\times \partial_B \Gamma$ is also ergodic for $\kappa\otimes \kappa$.
The proof is thus complete. \qed

\subsection{Geometric characterization of the Guivarc'h equality}\label{Section3implies4}
We finish the proof of Theorem~\ref{theoremBHM}.
We start proving the following proposition.

\begin{proposition}\label{(3)implies(4)BHM}
If the harmonic measure and the Patterson-Sullivan measure are equivalent, then $|d_G-vd_w|$ is uniformly bounded, where we recall that $d_G$ is the Green metric and $d_w$ is the word metric.
\end{proposition}

\begin{proof}
It follows from Proposition~\ref{(2)implies(3)BHM} and Corollary \ref{reflectedequivalence} that if $\kappa$ and $\nu$ are equivalent, their respective Radon-Nikodym are bounded away from 0 and infinity.
In particular, for any Borel set $A$ in the Bowditch boundary, we have
$\kappa(A)\asymp \nu(A)$, where the implicit constant does not depend on $A$.
The shadow lemmas for the Patterson-Sullivan and the harmonic measures show that $\kappa(\Omega(g))\asymp \mathrm{e}^{-vd_w(e,g)}$ and $\nu(\Omega(g))\asymp \mathrm{e}^{-d_G(e,g)}$, where the implicit constants do not depend on $g$.
It follows that $|d_G(e,g)-vd_w(e,g)|\leq C$ for some uniform $C$.
Since both distances are invariant by left multiplication, we have $|d_G(g,g')-vd_w(g,g')|\leq C$ for any $g,g'$.
\end{proof}

Recall that we want to prove that the following conditions are equivalent.
\begin{enumerate}
    \item The equality $h=lv$ holds.
    \item The measure $\nu$ is equivalent to the measure $\kappa$.
    \item The measure $\nu$ is equivalent to the measure $\kappa$ with Radon-Nikodym derivatives bounded from above and below.
    \item There exists $C\geq 0$ such that for every $g \in \Gamma$, $|d_G(e,g)-vd_w(e,g)|\leq C$.
\end{enumerate}
Proposition~\ref{(1)implies(2)BHM} shows that~(1) implies~(2), Proposition~\ref{(2)implies(3)BHM} shows that~(2) implies~(3) and Proposition~\ref{(3)implies(4)BHM} shows that~(3) implies~(4).
Finally, recall that $\omega_n$ is the $n$-th step of the random walk.
Then, $\frac{1}{n}d_w(e,\omega_n)$ almost surely converges to $l$ by definition of $l$.
As we saw above, according to \cite[Theorem~1.1]{BlachereHassinskyMathieu1}, $\frac{1}{n}d_G(e,\omega_n)$ almost surely converges to $h$.
This shows that~(4) implies~(1), hence Theorem~\ref{theoremBHM} is proved. \qed

\section{Virtually abelian parabolic subgroups}\label{Sectionabelianparabolics}
We first deduce from Theorem~\ref{theoremBHM} that if $h=lv$, then Ancona inequalities are satisfied along word geodesics.
We do not assume in this corollary that the parabolic subgroups are abelian and we will actually also use this result in the next section.

\begin{corollary}\label{coroh=lvimpliesAncona}
Let $\Gamma$ be a non-elementary relatively hyperbolic group.
Let $d_w$ be a word distance associated to a finite generating set.
Let $\mu$ be an admissible probability measure with finite super-exponential moment.
Assume that $h=lv$.
There exists some constant $C\geq 0$ such that the following holds.
Let $x,y,z$ be three points in this order on a word geodesic.
Then,
$$\frac{1}{C}G(x,y)G(y,z)\leq G(x,z)\leq CG(x,y)G(y,z).$$
\end{corollary}

\begin{proof}
Since $h=lv$, we have that $|d_{G}-vd_w|\leq c$ for some constant $c$,
where $d_G$ is the Green distance.
Let $x,y,z$ be three points in this order on a word geodesic.
Then,
$d_w(x,y)+d_w(y,z)=d_w(x,z)$, so that $|d_{G}(x,y)+d_{G}(y,z)-d_{G}(x,z)|\leq 3c$.
This translates into
$$\mathrm{e}^{-3c}\leq \frac{G(x,y)G(y,z)}{G(x,z)}\leq \mathrm{e}^{3c},$$
which is the desired inequality.
\end{proof}

The remainder of this section is devoted to the proof of Theorem~\ref{theoremhlvabelianparabolics}.
We fix a relatively hyperbolic group $\Gamma$ and assume that one of the parabolic subgroup is virtually abelian of rank at least 2.
We consider a word distance $d_w$ and an finitely supported admissible probability measure $\mu$.
We denote by $d_{G}$ the corresponding Green distance.

\subsection{The induced transition kernel on a parabolic subgroup}\label{Sectioninducedkernelparabolics}
Denote by $P$ a virtually abelian parabolic subgroup of rank at least 2.
Then, $P$ has a finite index subgroup which is isomorphic to $\Z^d$, $d\geq 2$.
Any section from $\Z^d$ to $P$ provides an identification $P=\Z^d\times \{1,...,N'\}$ as a set.
We say that a sequence $g_n=(z_n,k_n)$ in $P$ (where $z_n\in \Z^d$ and $k_n\in \{1,...,N'\}$) converges to a point in the geometric boundary of $P$ if $z_n$ tends to infinity and $\frac{z_n}{\|z_n\|}$ converges to a point in $\mathbb{S}^{d-1}$.
We denote by $\partial P$ the corresponding geometric boundary of $P$.

We will use results of \cite{DGGP}, where there is a precise description of sequences that converge to the Martin boundary of a relatively hyperbolic groups when \textit{all the parabolic subgroups are assumed to be virtually abelian}.
In this case, we define similarly the geometric boundaries of the parabolic subgroups.
We have the following.

\begin{theorem}\label{convergenceinMartin}\cite[Theorem~1.1]{DGGP}
Let $\Gamma$ be a relatively hyperbolic group with respect to virtually abelian subgroups.
Let $\mu$ be a probability measure on $\Gamma$ whose finite support generates $\Gamma$ as a semi-group.
A sequence $g_n$ converges to a point in the Martin boundary if and only if it converges to a conical limit point in the Bowditch boundary or there is a parabolic subgroup $H$ such that the projection of $g_n$ on $H$ converges to the geometric boundary $\partial H$ of $H$.
\end{theorem}

We only assume here that \textit{one} parabolic subgroup is virtually abelian of rank at least 2, so we cannot use directly this description of the Martin boundary.
However, an argument identical to the proof of \cite[Theorem~1.1]{DGGP} proves the following.

\begin{proposition}
Let $P$ be a parabolic subgroup. Let $g_n\in \Gamma$ be any sequence converging to the parabolic limit point $\xi$ fixed by $P$ in the Bowditch boundary. Then $g_n$ converges to a point in the Martin boundary if and only if its projection $\pi_P(g_n)$ on $P$ converges to a point in $\partial P$.
\end{proposition}

Let us give some details.
The proof of \cite[Theorem~1.1]{DGGP} is done dealing with parabolic limit points independently.
More precisely, \cite[Proposition~5.5]{DGGP} shows that if $\pi_P(g_n)$ converges to a point in $\partial P$, then $g_n$ and $\pi_P(g_n)$ both converge to the same point in the Martin boundary.
Conversely, assume that $g_n$ converges to a point in the Martin boundary and converges to $\xi$ in the Bowditch boundary.
Up to a sub-sequence, $\pi_P(g_n)$ also converges to a point in $\partial P$.
Using \cite[Proposition~5.5]{DGGP} again, we see that this subsequence also converges to the Martin boundary and the limit point is necessarily the same as the limit of $g_n$.
Now, two different points in the geometric boundary of $P$ give rise to two different points in the Martin boundary, according to \cite[Corollary~4.10]{DGGP}.
This proves that $\pi_P(g_n)$ only has one limit point in $\partial P$ and thus converges.

We recall some notations and results of \cite{Dussaule} and \cite{DGGP}.
Let $N_r(P)$ to be the $r$-neighborhood of $P$ and let $N(P)=N_r(P)$ for some $r$.
Consider the transition kernel $p_{\mu}$ of first return of the random walk to  $N(P)$. In other words, $p_{\mu}(g,g')$ is the probability that the random walk, starting at $g$ eventually returns to $N(P)$ and first does so at $g'$.
By definition, we have
$$p_{\mu}(g,g')=\sum_{n\geq 1}\sum_{g_1,...,g_{n-1}\notin N(P)}\mu(g^{-1}g_1)\mu(g_1^{-1}g_2)...\mu(g_{n-1}^{-1}g').$$
Then, $p_{\mu}$ can be considered as a $\Z^d$-invariant transition kernel on the state space $\Z^d\times \{1,...,N'\}$.
In \cite{Dussaule}, the first author proved some estimates for the Green function that we will use, provided the transition kernel on $\Z^d\times \{1,...,N'\}$ satisfies some technical assumptions.
As we will see these assumptions will be satisfied in our context.

For $g \in N(P)$, we write $g=(z,k)$, where $z\in \Z^d$ and $k\in \{1,...,N'\}$.
For simplicity, we will write $p_{\mu}(g,g')=p_{k,k'}(z,z')$, where $g=(z,k)$ and $g'=(z',k')$.
We define an $N'\times N'$-matrix $F(u),u\in \R^d$ as follows.
The coefficient at the $j$th row and $k$th column is given by
$$F_{j,k}(u)=\sum_{z\in \Z^d}p_{k,j}(0,z)\mathrm{e}^{u\cdot z}.$$

We denote by $\mathcal{F}\subset \R^d$ the interior of the subset of $\R^d$ such that every coefficient of $F(u)$ is finite.
If the random walk satisfies $\mu(e)>0$, then for every $u\in \mathcal{F}$, the matrix $F(u)$ is strongly irreducible and has a dominant eigenvalue $\lambda(u)$.
According to \cite[Lemma~3.20]{Dussaule}, if we change $\mu$ with $\tilde{\mu}=\frac{1}{2}\delta_e+\frac{1}{2}\mu$, then the corresponding Green function $\tilde{G}$ satisfies
$\tilde{G}=2G$.
In particular, $\tilde{G}$ still satisfies Ancona inequalities and since we will derive a contradiction with these inequalities, we can assume for simplicity that $\mu(e)>0$.
Since $\mu$ is admissible, if $\mu(e)>0$ then the corresponding random walk is aperiodic.
In particular, we can assume that the dominant eigenvalue $\lambda(u)$ of $F(u)$ is well defined.

As we will see, this eigenvalue encodes most of the behaviour of the Green function at infinity.
As explained above, in \cite{Dussaule}, there are two technical assumptions that we now state.
Let
$\mathcal{H}=\{u\in \mathcal{F},\lambda(u)=1\}$
and
$\mathcal{D}=\{u\in \mathcal{F},\lambda(u)\leq 1\}$.

\begin{hyp}\label{hyp1}
The set $\mathcal{D}$ is compact.
\end{hyp}

\begin{hyp}\label{hyp2}
The minimum of the function $\lambda$ is strictly smaller than 1.
\end{hyp}

\begin{lemma}\label{explicithomeom}\cite[Lemma~3.13]{Dussaule}.
Under assumptions~1 and~2, the set $H$ is homeomorphic to $\mathbb{S}^{d-1}$, an explicit homeomorphism is given by
$$u\in \mathcal{H}\mapsto \frac{\nabla \lambda (u)}{\|\nabla \lambda (u)\|}\in \mathbb{S}^{d-1}.$$
\end{lemma}

Denote by $G_p$ the Green function associated to $p$.
For simplicity, we use the notation $G_{j,j'}(z,z')=G_p(g,g')$ whenever $g=(z,j)$ and $g'=(z',j')$.
The crucial estimates we will need are given by the following proposition.
If $v\in \R^d$, we denote by $\langle v\rangle$ the closest point from $v$ in $\Z^d$.
If there are several choices, we make one arbitrarily.
We will be interested in points in $\Z^d$ of the form $\langle t\nabla \lambda(u)\rangle$, where $u\in H$ and $t\in \R$.

\begin{proposition}\label{estimatesGreen}\cite[Proposition~3.15]{Dussaule}
Assume that $p$ satisfies Assumptions~1 and~2.
Then, for every $z\in \Z^d$ and every $u\in \mathcal{H}$, we have that
$$(2\pi t)^{\frac{d-1}{2}}G_{j,k}(z,\langle t\nabla \lambda (u)\rangle)\mathrm{e}^{u\cdot (z-\langle t\nabla \lambda (u)\rangle)}\underset{t\rightarrow +\infty}{\longrightarrow}C(u),$$
where $C(u)$ depends continuously on $u$.
Moreover, if $z$ if fixed, then the convergence if uniform in $u\in \mathcal{H}$.
\end{proposition}

\subsection{Estimates for the Green function along a parabolic subgroup}\label{SectionEstimatesGreenalongparabolics}
Recall that $P$ is a virtually abelian parabolic subgroup of rank at least 2.
It is proved in \cite{DGGP} that if the fixed neighborhood $N(P)$ of $P$ is large enough, Assumptions~1 and~2 are satisfied for the transition kernel of first return to $N(P)$, see Proposition~4.8, Lemma~4.9 and Lemma~5.9, Proposition~5.10 there.
We can thus apply Proposition~\ref{estimatesGreen} above.

Consider a geodesic ray $[e,\xi)$ in the Cayley graph of $\Gamma$ that stays inside $N(P)$.
Such a geodesic always exists, if the neighborhood $N(P)$ of $P$ is large enough.
Indeed, according to \cite[Lemma~4.3]{DrutuSapir}, any geodesic from $e$ to $p\in P$ is contained in a fixed neighborhood of $P$ that does not depend on $p$.
We can thus construct $[e,\xi)$ as a limit of geodesics from $e$ to $p_n$, when $p_n$ goes to infinity.

Let $g_n$ be a sequence on $[e,\xi)$ with the property that for every $n\geq m$, we have $d(e,g_n)\geq d(e,g_m)$.
By definition, $g_n$ converges to the parabolic limit point fixed by $P$ in the Bowditch boundary.
Up to taking a subsequence, we can assume that $g_n$ converges to a point in the Martin boundary.
Thus, according to Theorem~\ref{convergenceinMartin} the projection $z_n$ of $g_n$ onto $\Z^d$ satisfies that $\frac{z_n}{\|z_n\|}$ converges to some point $\theta\in \mathbb{S}^{d-1}$.

Fix $l_0$ and let $g=g_{l_0}=(z,j)$.
For $n\geq l_0$, write $g_n=(z_n,k_n)$.
Up to taking a subsequence, we can assume that $k_n$ is a constant $k_n=k$.
Let $\theta_n=\frac{z_n}{\|z_n\|}$ so that $\theta_n$ converges to $\theta$.
Let $u_n\in \mathcal{H}$ be such that $\theta_n=\frac{\nabla\lambda(u_n)}{\|\nabla\lambda(u_n)\|}$.
According to Lemma~\ref{explicithomeom}, $u_n$ is well defined and converges to some $u\in\mathcal{H}$ such that $\theta=\frac{\nabla\lambda(u)}{\|\nabla\lambda(u)\|}$.
Define then $t_n=\frac{\|z_n\|}{\|\nabla \lambda(u_n)\|}$.
By definition,
$t_n\nabla \lambda(u_n)=\|z_n\|\theta_n=z_n$.
Since $u_n$ converges to $u$ and $\nabla \lambda$ is continuous (see \cite[Lemma~3.3]{Dussaule}), we deduce from Proposition~\ref{estimatesGreen} that
$\|z_n\|^{\frac{d-1}{2}}G_{j,k}(z,z_n)\mathrm{e}^{u\cdot (z_n-z)}$ converges to some limit that only depends on $\theta$.

Finally, \cite[Lemma~5.7]{DGGP} shows that the Green function associated to the transition kernel is the restriction to $N(P)$ of the Green function associated with $\mu$ on the whole group.
Thus, we have that
$$\|z_n\|^{\frac{d-1}{2}}G(g,g_n)\mathrm{e}^{u\cdot (z_n-z)}$$
converges to some limit that does not depend on $g$, for some $u\in \R^d$ that only depends on $\theta\in \mathbb{S}^{d-1}$.
We can choose $m_1$ large enough so that
$$G(g,g_{m_1})\asymp \frac{\mathrm{e}^{u\cdot (z_{m_1}-z)}}{\|z_{m_1}\|^{\frac{d-1}{2}}}.$$
Now, if $m_1$ is fixed, we can choose $n_1$ large enough so that
$$G(g,g_{n_1})\asymp \frac{\mathrm{e}^{u\cdot (z_{n_1}-z)}}{\|z_{n_1}\|^{\frac{d-1}{2}}}$$
and
$$G(g_{m_1},g_{n_1})\asymp \frac{\mathrm{e}^{u\cdot (z_{n_1}-z_{m_1})}}{\|z_{n_1}\|^{\frac{d-1}{2}}}.$$
Thus we can find sequences $m_k$ and $n_k$ tending to infinity with $l_0\leq m_k\leq n_k$ (where $n_k$ depends on $m_k$) such that
$$\frac{G(g,g_{m_k})G(g_{m_k},g_{n_k})}{G(g,g_{n_k})}\asymp \frac{1}{\|z_{m_k}\|^{\frac{d-1}{2}}}.$$
Since $d\geq 2$, the right-hand side converges to 0 when $k$ tends to infinity, so we get a contradiction to Corollary~\ref{coroh=lvimpliesAncona}.
This proves Theorem~\ref{theoremhlvabelianparabolics}. \qed

\medskip
Corollary~\ref{coroabelianparabolics+GMM} follows from Theorem~\ref{theoremhlvabelianparabolics}.
Indeed, if $h=lv$, then every parabolic subgroup is abelian of rank at most 1.
In particular, $\Gamma$ is hyperbolic.
We can thus apply the results of Gou\"ezel, Math\'eus and Maucourant.
Precisely, \cite[Theorem~1.5]{GMM} shows that $\Gamma$ is virtually free. \qed

\section{Free products of nilpotent groups}\label{Sectionnilpotentparabolics}
We first derive from Ancona inequalities proved in Corollary~\ref{coroh=lvimpliesAncona} the following property.

\begin{proposition}\label{propfinitemiddles}
Let $\Gamma$ be a non-elementary relatively hyperbolic group.
Let $d_w$ be a word distance associated to a finite generating set.
Let $\mu$ be an admissible probability measure with finite super-exponential moment.
Assume that $h=lv$.
Then there exists $N\geq 0$ such that for any $x,y\in \Gamma$, there are most $N$ points in $\Gamma$ which are middle points of some geodesic between $x$ and $y$.
\end{proposition}

\begin{remark}
If the geodesic from $x$ to $y$ is of odd length $2n+1$, what we call a middle point is a point at distance $n$ from $x$.
This choice is arbitrary and will have no importance in the following.
\end{remark}

\begin{proof}
Let $x,y\in \Gamma$ and assume that $w$ is on a geodesic from $x$ to $y$.
Thus, by Ancona inequalities (Corollary \ref{coroh=lvimpliesAncona}),
$G(x,y)\leq CG(x,w)G(w,y)$.
We can rewrite this as
$$\frac{P_x(\text{going to }y,\text{ passing through }w)}{P_x(\text{going to }y)}\geq \frac{1}{c_1}.$$
Denote by $A_w$ the event $(\text{passing through }w)$.
Conditioning on the event $E_y=(\text{going to }y)$,
we thus have
$$P_x(A_w|E_y)\geq \frac{1}{c_1}.$$
This is true for any $w$ which is on a geodesic from $x$ to $y$.
Denote now $k=\lceil c_1 \rceil$ and assume that there are $w_1,...,w_{k+1}$ different points that are on (possibly different) geodesics from $x$ to $y$.
Then, the sum of the probabilities $P_x(A_{w_j}|E_y)$ is greater than one, so at least two of those events happen at the same time.
We can find a bound $c_2$, not depending on $x$ and $y$, such that
there are at least two points $w_i$ and $w_j$ such that
$P_x(A_{w_j}\cap A_{w_i}|E_y)\geq c_2$.

We assume that the $w_i$, $1\leq i \leq k+1$ are midpoints of geodesics from $x$ to $y$.
We proved that there are at least two points $w_i$ and $w_j$ such that
$$\frac{P_x(\text{going to }w_i, \text{ then to }w_j, \text{ then to }y)}{P_x(\text{going to }y)}\geq \frac{c_2}{2},$$
that is
\begin{equation}\label{finitemiddles1}
    \frac{G(x,w_i)G(w_i,w_j)G(w_j,y)}{G(x,y)}\geq c_3.
\end{equation}
Since $d(x,w_i)+d(w_j,y)=d(x,y)$ and since $|d_{G}-vd_w|$ is bounded, we have
\begin{equation}\label{finitemiddles2}
   G(x,w_i)G(w_j,y)\asymp G(x,y).
\end{equation}
Combining~(\ref{finitemiddles1}) and~(\ref{finitemiddles2}), we get
$$G(x,y)G(w_i,w_j)\geq G(x,y)c_3,$$ so that
$G(w_i,w_j)\geq c_3$.
In other words, $d_G(w_i,w_j)\leq c_4$, and thus $d_w(w_i,w_j)\leq C$.

To sum up,
if there are more than $c_1$ midpoints of geodesics from $x$ to $y$, then two of them are within a uniform distance $C$ of each other.
Since balls of a fixed radius have a uniform finite number of elements, there are a finite number of such middle points.
\end{proof}

The remainder of this section is devoted to the proof of Theorem~\ref{hlvnilpotentfreefactors}.
Let $\Gamma=\Gamma_1*\cdots *\Gamma_N$ be a free product of finitely generated groups.
Assume that $\Gamma_1$ is a non-virtually cyclic nilpotent group.
Choose a finite generating set $S_i$ for each free factor $\Gamma_i$ and define $S=\cup S_i$,
so that $S$ is an adapted generating set for $\Gamma$.

We will use Proposition~\ref{propfinitemiddles} and argue by contradiction.
Our goal is to construct points $x$ and $y$ in $\Gamma_1$ such that the set of midpoints of geodesics from $x$ to $y$ is arbitrarily large.
We will use a construction of Walsh in \cite{Walsh}.

We first recall some facts about finitely generated nilpotent groups.
Denote by $\Gamma_1^1=\Gamma_1$, $\Gamma_1^2=[\Gamma_1,\Gamma_1]$, ..., $\Gamma_1^n=[\Gamma_1^{n-1},\Gamma_1]$.
Let $N_{\Gamma_1}$ be the nilpotency step of $\Gamma_1$, that is $N_{\Gamma_1}$ is the last $n$ such that $\Gamma_1^n$ is not trivial.
The groups $\Gamma_1^n/\Gamma_1^{n+1}$ are finitely generated abelian groups.
The homogeneous dimension of $\Gamma_1$ is defined as
$$d_{\Gamma_1}=\sum_{n=1}^{N_{\Gamma_1}}n\ \mathrm{rank}\left (\Gamma_1^{n}/\Gamma_1^{n+1}\right ).$$

\begin{lemma}\label{growthandhomogeneousdimension}
\begin{enumerate}[a)]
    \item \cite[Lemma~4.6]{Raghunathan} Any finitely generated nilpotent group has a finite index subgroup which is torsion free.
    \item \cite[Proposition~(5)]{Pansu} A finitely generated nilpotent group has polynomial growth. More precisely the cardinality of a ball of radius $n$ is asymptotic to $Cn^d$, where $d$ is the homogeneous dimension.
\end{enumerate}
\end{lemma}

\begin{remark}
Point b) has a famous converse due to Gromov: a finitely generated group has polynomial growth if and only if it is virtually nilpotent. The identification of the degree as the homogeneous dimension defined above is due to Guivarc'h and Bass independently.
\end{remark}

In particular, in our situation we have a subgroup $\tilde{\Gamma}_1$ in $\Gamma_1$ which is torsion free and has finite index in $\Gamma_1$.
Let $\tilde{\Gamma}_1^1=\tilde{\Gamma}_1$, $\tilde{\Gamma}_1^2=[\tilde{\Gamma}_1,\tilde{\Gamma}_1]$, ..., $\tilde{\Gamma}_1^n=[\tilde{\Gamma}_1^{n-1},\tilde{\Gamma}_1]$.
We have the following lemma.

\begin{lemma}\label{finiteindexsubgroups}\cite[Lemma~3.1]{Walsh}
For every $n$, $\tilde{\Gamma}_1^n$ has finite index in $\Gamma_1^n$.
\end{lemma}

Actually, \cite[Lemma~3.1]{Walsh} states that $\tilde{\Gamma}_1^2$ has finite index in $\Gamma_1^2$, but the proof gives the result for all $n$.

We now consider the abelianization $\Gamma_1/[\Gamma_1,\Gamma_1]$ of $\Gamma_1$.
By definition, it is finitely generated abelian, so it is of the form $\Z^d\times T$, where $T$ is a finite group (the torsion group) and $d$ is the rank of $\Gamma_1/[\Gamma_1,\Gamma_1]$.
\begin{lemma}\label{rankgeq2abelianization}
With these notations, $d\geq 2$.
\end{lemma}

\begin{proof}
We first show that the rank of $\Gamma_1/[\Gamma_1,\Gamma_1]$ is also the rank of $\tilde{\Gamma}_1/[\tilde{\Gamma}_1,\tilde{\Gamma}_1]$.
Indeed, $\tilde{\Gamma}_1$ has finite index in $\Gamma_1$, so $\tilde{\Gamma}_1/[\tilde{\Gamma}_1,\tilde{\Gamma}_1]$ has finite index in $\Gamma_1/[\tilde{\Gamma}_1,\tilde{\Gamma}_1]$. Thus, they have the same growth.
Since $[\tilde{\Gamma}_1,\tilde{\Gamma}_1]$ is a subgroup of $[\Gamma_1,\Gamma_1]$, the growth of $\Gamma_1/[\Gamma_1,\Gamma_1]$ is smaller or equal to the growth of $\Gamma_1/[\tilde{\Gamma}_1,\tilde{\Gamma}_1]$.
This proves that
$$\mathrm{rank}\left (\Gamma_1/[\Gamma_1,\Gamma_1]\right )\leq \mathrm{rank}\left (\tilde{\Gamma}_1/[\tilde{\Gamma}_1,\tilde{\Gamma}_1]\right ).$$
Similarly, $[\tilde{\Gamma}_1,\tilde{\Gamma}_1]$ has finite index in $[\Gamma_1,\Gamma_1]$ so that $[\tilde{\Gamma}_1,\tilde{\Gamma}_1]/[[\tilde{\Gamma}_1,\tilde{\Gamma}_1],\tilde{\Gamma}_1]$ has the same growth as $[\Gamma_1,\Gamma_1]/[[\tilde{\Gamma}_1,\tilde{\Gamma}_1],\tilde{\Gamma}_1]$ and so
$$\mathrm{rank}\left ([\Gamma_1,\Gamma_1]/[[\Gamma_1,\Gamma_1],\Gamma_1]\right )\leq \mathrm{rank}\left ([\tilde{\Gamma}_1,\tilde{\Gamma}_1]/[[\tilde{\Gamma}_1,\tilde{\Gamma}_1],\tilde{\Gamma}_1]\right ).$$
Using Lemma~\ref{finiteindexsubgroups}, an induction argument yields for all $n$
\begin{equation}\label{rankcomparison}
    \mathrm{rank}\left (\Gamma_1^{n}/\Gamma_1^{n+1}\right )\leq \mathrm{rank}\left (\tilde{\Gamma}_1^{n}/\tilde{\Gamma}_1^{n+1}\right ).
\end{equation}
Since $\tilde{\Gamma}_1$ has finite index in $\Gamma_1$, they have the same growth. Thus, according to Lemma~\ref{growthandhomogeneousdimension}, all inequalities given by~(\ref{rankcomparison}) are equalities.
In particular, we get that $\mathrm{rank}\left (\Gamma_1/[\Gamma_1,\Gamma_1]\right )=\mathrm{rank}\left (\tilde{\Gamma}_1/[\tilde{\Gamma}_1,\tilde{\Gamma}_1]\right )$.

We now argue by contradiction and assume that $d\leq 1$.
Using the work of Mal'cev \cite{Malcev}, $\tilde{\Gamma}_1$ is torsion free nilpotent, so it is isomorphic to a lattice in a simply connected nilpotent Lie group $\mathcal{N}$ (see also \cite[Theorem~2.18]{Raghunathan}).
Denote by $\mathfrak{n}$ the Lie algebra of $\mathcal{N}$
and write
$\mathfrak{n}=[\mathfrak{n},\mathfrak{n}]\oplus \mathfrak{n}_1$ as a vector space.
Such a decomposition is not necessarily compatible with the Lie algebra structure.
However, the rank of $\tilde{\Gamma}_1/[\tilde{\Gamma}_1,\tilde{\Gamma}_1]$ is the dimension of $\mathcal{N}/[\mathcal{N},\mathcal{N}]$ (see \cite[Theorem~2.10]{Raghunathan}).
Since $d\leq 1$, we thus have that $\mathfrak{n}_1$ has dimension 0 or 1.
We show that in both cases, $\mathcal{N}$ is abelian.

Let $u,v\in \mathfrak{n}$.
Write $u=u_0+u_1$ and $v=v_0+v_1$, where $u_0,v_0\in [\mathfrak{n},\mathfrak{n}]$ and $u_1,v_1\in \mathfrak{n}_1$.
Then $[u,v]=[u_0,v_0]+[u_0,v_1]+[u_1,v_0]+[u_1,v_1]$.
Since $\mathfrak{n}_1$ is zero or one-dimensional, $[u_1,v_1]=0$, thus $[x,y]\in [\mathfrak{n},[\mathfrak{n},\mathfrak{n}]]$.
An induction argument then shows that $[\mathfrak{n},\mathfrak{n}]\subset [\mathfrak{n},[\mathfrak{n},\mathfrak{n}]]\subset \cdots \subset [\mathfrak{n},\dots [\mathfrak{n},\mathfrak{n}]]=\{0\}$.
We deduce from this that $\mathcal{N}$ is abelian, so that $\tilde{\Gamma}_1$ is abelian.
In particular, $\tilde{\Gamma}_1/[\tilde{\Gamma}_1,\tilde{\Gamma}_1]$ is just $\tilde{\Gamma}_1$, which is thus abelian of rank at most 1.
In particular, $\Gamma_1$ is virtually cyclic, which is a contradiction.
\end{proof}

We consider the projection map $\phi :\Gamma_1\rightarrow \Z^d$.
Recall that $S_1$ is a finite generating set of $\Gamma_1$.
Then, $\phi(S_1)$ is a finite generating set of $\Z^d$.
We see $\Z^d$ as embedded in $\R^d$ in the standard way.
Let $P=\Conv (\phi(S_1))$ be the convex hull of the image of $\phi(S_1)$ in $\R^d$.
Then $P$ is a polytope.
A facet of $P$ is a proper face of maximal dimension.
We have the following results.

\begin{lemma}\label{lemma4.1Walsh}\cite[Lemma~4.1]{Walsh}
Let $V$ be a subset of $S_1$ such that $\phi(V)$ is contained in a facet of $P$.
Then, any word with letters in $V$ is a geodesic for the word metric associated with $S_1$.
\end{lemma}

\begin{lemma}\label{lemma4.2Walsh}\cite[Lemma~4.2]{Walsh}
Let $H$ be a nilpotent group generated as a group by a finite set $V$.
Then, for any $g\in H$, there exist $x$ and $y$ in $H$ written with letters in $V$ such that $gx=y$.
\end{lemma}

Actually, in \cite[Section~4]{Walsh}, the author fixes a symmetric finite generating set, but the assumption of symmetry is not needed for these two lemmas.
This is obvious for Lemma~\ref{lemma4.2Walsh}, since the statement does not involve $S_1$ nor $\Gamma_1$.
For Lemma~\ref{lemma4.1Walsh}, this readily follows from the proof of \cite[Lemma~4.1]{Walsh}, which we recall for convenience.

\begin{proof}[Proof of Lemma~\ref{lemma4.1Walsh}]
Let $F\subset \R^d$ be the facet into which $V$ is mapped and let $f:\R^n\rightarrow \R$ be the linear function defined by $f(u)=1$ for $u\in F$.
Let $x_1...x_n\in \Gamma_1$ be an element written with elements of $V$.
Then, $f\circ  \phi(x_1...x_n)=n$.
Assume that $x_1...x_n=y_1...y_m$ with $y_i\in S_1$.
Then, since $f(u)\leq 1$ for $u\in P$,
$f\circ \phi(y_1...y_m)\leq m$.
Hence, we necessarily have $m\geq n$ and so $x_1,...,x_n$ is a geodesic.
\end{proof}

We now fix a facet $F$ and $V\subset S_1$ such that $\phi(V_1)=F$.
Our goal is to construct elements $x,y$ such that
\begin{enumerate}
    \item $x$ and $y$ commute,
    \item for every non-negative $k,l,m$, we can write $x^ky^lx^m$ as a geodesic with elements of $V$,
    \item if $x^ky^l=x^{k'}y^{l'}$, then $k=k'$ and $l=l'$.
\end{enumerate}

Denote by $\Gamma_V$ the subgroup of $\Gamma_1$ generated by $V$.
Then, $\Gamma_V$ is also nilpotent.
There are two cases. Either $[\Gamma_V,\Gamma_V]$ is finite or it is not.
We first consider the case where it is infinite.
Let $\Gamma_V^1=\Gamma_V$, $\Gamma_V^2=[\Gamma_V,\Gamma_V]$,...,$\Gamma_V^n=[\Gamma_V,\Gamma_V^{n-1}]$.

Any element of $\Gamma_V$ mapped to a non-trivial element of $\Z^d$ by $\phi$ has infinite order.
Consider the last $n$ such that $\Gamma_V^n$ has an infinite order element and consider such an element $g\in \Gamma_V^n$.
Then, $\Gamma_V^{n+1}$ has only torsion elements.
Since it is nilpotent, it is finite.
In particular, we cannot have $n=1$ since $[\Gamma_V,\Gamma_V]$ is infinite.
According to Lemma~\ref{lemma4.2Walsh}, there exist $x$ and $y$ written with elements of $V$ such that $gx=y$.
If $x$ were trivial, then we would have $g=y$.
In particular, $y$ could not be trivial so it would mapped to the facet $F$ and would also be in $[\Gamma_1,\Gamma_1]$, which is absurd.
Thus, $x$ cannot be trivial.
Since $g^{-1}y=x$, we show similarly that $y$ is not trivial.

\begin{lemma}\label{commuteuptopower}
There exists $m\geq1$ such that $g$ and $x^m$ commute.
\end{lemma}

\begin{proof}
Assume that $xgx^{-1}g^{-1}\neq e$.
Then $x^2gx^{-2}g^{-1}\neq xgx^{-1}g^{-1}$.
Otherwise, we would have
$x^2gx^{-2}g^{-1}= xgx^{-1}g^{-1}$
so that $xgx^{-1}=g$ and thus $xgx^{-1}g^{-1}=e$.
We also have similarly $x^3gx^{-3}g^{-1}\neq x^2gx^{-1}g^{-1}$.
Now, if $x^2gx^{-2}g^{-1}\neq e$, a similar argument shows that
$x^3gx^{-3}g^{-1}\neq xgx^{-1}g^{-1}$.
By induction, we show that if $x^kgx^{-k}g^{-1}\neq e$ for $k=1,...,m$, then $x^{m+1}gx^{-(m+1)}g^{-1}$ is different from all the $x^kgx^{-k}g^{-1}$.
Since all these elements lie in $\Gamma_V^{n+1}$ which is finite, this concludes the proof.
\end{proof}

We write $x'=x^m$ and $y'=yx^{m-1}$.
Up to replacing $x$ with $x'$ and $y$ with $y'$, we thus have that $gx=y$, $x$ and $g$ commute, hence $x$ and $y$ commute, and $x$ and $y$ are written with elements of $V$.
According to Lemma~\ref{lemma4.1Walsh}, the sequence $v_1,...,v_p$ of elements of $V$ such that $x=v_1...v_p$ yields a geodesic from $e$ to $x$ and similarly with $y$.
Moreover, any concatenation of these geodesics from $e$ to $x$ and from $e$ to $y$ is again a geodesic.
In particular, for any $k,l,m$, we can write $x^{k}y^{l}x^m$ with elements of $V$ and this yields a geodesic from $e$ to $x^ky^lx^m$.

\begin{lemma}\label{differentmiddles}
If $x^ky^l=x^{k'}y^{l'}$, then $k=k'$ and $l=l'$.
\end{lemma}

\begin{proof}
We have $x^{k-k'}=y^{l'-l}=g^{l'-l}x^{l'-l}$.
Projecting this equation in $\Z^d$, we get
$\phi(x)^{k-k'}=\phi(x)^{l'-l}$.
Since $\phi(x)\neq 0$ and there is no torsion in $\Z^d$, we have $k-k'=l'-l$.
Thus, $e=g^{l'-l}$ and since $g$ has infinite order, $l=l'$, hence $k=k'$.
\end{proof}

Assume now that $[\Gamma_V,\Gamma_V]$ is finite.
This time, we choose $x$ and $y$ as two elements of $V$ that are sent to two different elements of the facet $F$.
This is possible since $d\geq 2$, according to Lemma~\ref{rankgeq2abelianization}.
Since $\phi(x)$ and $\phi(y)$ are linearly independent elements of $\Z^d$, we still have the conclusion of Lemma~\ref{differentmiddles}.
The same argument as in the proof of Lemma~\ref{commuteuptopower} shows that there is $m\geq 1$ such that $x^m$ and $y$ commute.
Up to replacing $x$ with $x^m$ (which is still linearly independent of $y$, once mapped to $\Z^d$),
we still get that
\begin{enumerate}
    \item $x$ and $y$ commute,
    \item we can write $x^ky^lx^m$ as a geodesic with elements of $V$,
    \item if $x^ky^l=x^{k'}y^{l'}$, then $k=k'$ and $l=l'$.
\end{enumerate}

Whether $[\Gamma_V,\Gamma_V]$ is finite or not, we use such elements $x$ and $y$
to finish the proof of Theorem~\ref{hlvnilpotentfreefactors}.
We argue by contradiction and assume that $h=lv$.
Proposition~\ref{propfinitemiddles} shows that for any $g_0\in \Gamma$, there is a uniformly finite number of middle points on a geodesic from $e$ to $g_0$.

According to Condition~(2) above, we can find $x$ and $y$ such that $g_{k,l,m}=x^ky^lx^m$, written with elements of $V$ yields a geodesic from $e$ to $g_{k,l,m}$.
We denote this geodesic by $\alpha$.
If $k+m=n$ is fixed, we can find $l$ large enough so that the length of $y^l$ is larger than the length of $x^{k+m}$.
Thus, for every $k,m$ such that $k+m=n$, the middle point on the geodesic $\alpha$ is on the sub-segment from $x^k$ to $x^ky^l$.

Letting $k$ vary, we thus get several middle points on several geodesics from $e$ to $g_{k,l,m}$.
Since $x$ and $y$ commute, these geodesics have the same endpoints.
Those middle points are of the form
$y^{l_0},xy^{l_1},x^2y^{l_2},...,x^ny^{l_n}$.
According to Condition~(3) above, all those middle points are different.
Up to enlarging $n$ if necessary, we obtain a contradiction with Proposition~\ref{propfinitemiddles}. \qed

\section{Groups with connected set of conical limit points}\label{Sectionconnectedconicallimitpoints}
In this section, we prove Theorem~\ref{sphericalBowditchboundary}. We will actually prove a stronger statement (see Theorem~\ref{mainTheoremconnected}).
We consider a non-elementary relatively hyperbolic group $\Gamma$ and a probability measure $\mu$ on $\Gamma$.
We consider a word metric $d_w$ and assume that $\mu$ has finite super-exponential moment with respect to $d_w$.

\subsection{Continuity of the Naim kernel}
We prove here that the Naim kernel $\Theta(\xi,\zeta)$ defined above is continuous at all pairs of distinct conical points  $(\xi,\zeta)$ of $\partial_{B}\Gamma$.
Recall that
$$\Theta(\xi,\zeta)=\liminf_{g\in \Gamma \to \xi}\lim_{h\in \Gamma \to \zeta}\frac{G(g,h)}{G(g,e)G(e,h)}=\liminf_{g\in \Gamma\to \xi}\frac{K_{\zeta}(g)}{G(g,e)}.$$
Notice that $-\frac{1}{2}\log \Theta$ is the analogue of the Gromov product associated with the Green metric.
Beware that in general, the Gromov product is not continuous with respect to both variables on the Gromov boundary, see \cite{NicaSpakula}.
Our proof is very specific to the Naim kernel and the Green distance.
If $A\subset \Gamma$, we denote by $G(x,y;A)$ the Green function at $x$ and $y$ conditioned by staying in $A$.
Precisely,
$$G(x,y;A)=\sum_{n\geq 0}P(\omega_0=x,\omega_n=y,\omega_k\in A,1\leq k \leq n-1),$$
where we recall that $\omega_k$ is the $k$-th step distribution of the random walk.
We begin by recalling the following enhanced version of relative Ancona inequalities which is obtained by combining Proposition~\ref{Floydgeo} and \cite[Theorem 5.2]{GGPY}.

\begin{proposition}\label{weakAnconastrongform}
For each $\epsilon\geq \epsilon_0$, $\eta>0$, $D>0$, $c>0$, there is an $R>0$ such that if $x,y,z\in \Gamma$ and $y$ is within $D$ of an $(\epsilon,\eta)$-transition point on a word geodesic from $x$ to $z$ then $G(x,z;B_{R}(y)^c)\leq c G(x,z)$,
where $B_R(y)^c$ is the complementary of the ball of center $y$ and radius $R$.
\end{proposition}

\begin{proposition}\label{ratiosconverge}
Suppose $(x_n,x'_n,y_n,y'_n)\in \Gamma^4$ is a sequence of quadruples with the following property. 
There exist $D>0,\epsilon \geq \epsilon_0,\eta>0$, 
and an infinite geodesic $\alpha$ with $(\epsilon, \eta)$-transition points $p_n,q_n\in \Gamma$ with $d_{w}(p_n,q_n)\to \infty$ such that any geodesic segment $[x_{n},y_{n}]$ and $[x'_{n},y'_{n}]$ passes within $D$ of both $p_n$ and $q_n$.
Then 
$$\frac{G(x_{n},y_{n})G(x'_{n},y'_{n})}{G(x'_{n},y_{n})G(x_{n},y'_{n})}\underset{n\to \infty}{\longrightarrow}1.$$
\end{proposition}

\begin{proof}
Suppose by contradiction that $x_n,y_n,x'_n,y'_n\in \Gamma$ are sequences with the stated property for which the conclusion fails. Then there is a subsequence for which the quotient $\frac{G(x_{n},y_{n})G(x'_{n},y'_{n})}{G(x'_{n},y_{n})G(x_{n},y'_{n})}$ converges to some $L\neq 1$ (which can be infinite).
Any sub-subsequence of this subsequence will also satisfy that the above quotient converges to $L\neq 1$.
Thus, up to taking a subsequence, we can assume without loss of generality that $(x_n,y_n,x'_n,y'_n)\to (\xi,\zeta,\xi',\zeta')$ in the Bowditch compactification and $(x_n,y_n,x'_n,y'_n) \to (\tilde{\xi},\tilde{\zeta},\tilde{\xi}',\tilde{\zeta}')$ in the Martin compactification, while the above quotient does not converge to 1.
In addition, passing to a subsequence we may assume either $p_n$ or $q_n$ (say $q_n$) converge to a conical point of $\partial_B \Gamma$. This point must then equal the limit $\zeta$ of the $y_n$ and $y'_n$.
It follows that $\zeta=\zeta'$ and $\tilde{\zeta}=\tilde{\zeta}'$.

Let $c>0$ be a small number.
According to Proposition~\ref{weakAnconastrongform},
there is an $R>0$ large enough so that $G(x,y;B_{R}(w)^c)<cG(x,y)$ whenever $x,y,w\in \Gamma$ with $w$ within $2D$ of $\Tr_{\epsilon,\eta}[x,y]$.
By our assumption on the sequences and the relatively thin triangles property, for each $k>0$ and all large enough $n$, there exist $p,q$ on the geodesic ray or bi-infinite geodesic $(\xi,\zeta)$ with $d_{w}(p,q)>k$ such that any geodesic segment $[x_{n},y_{n}]$,$[x'_{n},y_{n}]$,$[x_{n},y'_{n}]$, $[x'_{n},y'_{n}]$ contains a $(\epsilon,\eta)$-transition points within $2D$ of $p$ and $q$.
Paths from $x_n$ to $y_n$ are of three types:
\begin{enumerate}[a)]
    \item Paths that miss one of $B_{R}(p)$ or $B_{R}(q)$.
    \item Paths that pass both $B_{R}(p)$ and $B_{R}(q)$, and whose last intersection with $B_{R}(p)$ comes after their last intersection with $B_{R}(q)$.
    \item Paths that pass both $B_{R}(p)$ and $B_{R}(q)$, and whose last intersection with $B_{R}(q)$ comes after their last intersection with $B_{R}(p)$.
\end{enumerate}

Paths of type a) contribute at most $2cG(x_{n},y_{n})$ to $G(x_{n},y_{n})$.
We will prove that the contribution of paths of type b) also is small compared to $G(x_n,y_n)$.
The contribution of such paths can be written as
$$\sum_{u\in B_{R}(p),v\in B_{R}(q)}G(x_{n},v)G(v,u;B_{R}(q)^{c}\cap B_{R}(p)^{c})G(u,y_{n})$$
which can be bounded by
$$\sum_{u\in B_{R}(p),v\in B_{R}(q)}G(x_{n},v)G(v,u)G(u,y_{n})$$
Recall that $d_{G}$ and $d_{w}$ are quasi-isometric (see \cite[Lemma~4.2]{Haissinsky}), so there exists a constant $t\in (0,1)$ such that for every $g_1,g_2,g_3\in \Gamma$, we have $G(g_1,g_2)\leq t^{d_{w}(g_1,g_2)}$ and 
$G(g_1,g_2)/G(g_1,g_3)\leq t^{-d_{w}(g_2,g_3)}$.
In particular, for every $u\in B_{R}(p),v\in B_{R}(q)$ we have
$$G(v,u)\leq t^{d_{w}(u,v)}\leq t^{k-2R}$$ 
$$G(x_{n},v)\leq t^{-d_w(v,q)}G(x_n,q)\leq t^{-R}G(x_n,q)$$ and
$$G(u,y_n)\leq t^{-d_{w}(u,p)}G(p,y_n)\leq t^{-R}G(p,y_n).$$
Furthermore, the number of points in a Cayley ball of radius $R$ is at most $Ce^{vR}$ for a constant $C>1$.
Thus, the contribution of paths of type b) is bounded above by 
$$C^{2}e^{2vR}t^{-4R}G(x_n,q)G(q,p)G(p,y_n).$$

By relative Ancona inequalities~(\ref{relativeAncona}), we have
$$G(x_n,q)\leq AG(x_{n},p)G(p,q),\hspace{.2cm}G(p,y_n)\leq AG(p,q)G(q,y_n)$$ where $A>1$ depends only on $\epsilon,\eta,D$.
Consequently the contribution of paths of type b) is bounded above by 

\begin{align*}
&C^{2}A^{2}e^{2vR}t^{-4R}G(x_n,p)G(p,q)G(q,p)G(p,q)G(q,y_n)\\
&\leq C^{2}A^{2}e^{2vR}t^{-4R}G(x_n,y_n)G(p,q)G(q,p)\\
&\leq C^{2}A^{2}e^{2vR}t^{-4R} t^{2k}G(x_{n},y_{n})
\end{align*}
which is smaller than $cG(x_n,y_n)$ if $k$ is chosen large enough.

Thus, the contribution of paths of type c) is at least $(1-3c)G(x_n,y_n)$.
Furthermore, conditioning by the last visit to $B_R(p)$ and the first subsequent visit to  $B_R(q)$, we can express this contribution as
$$\sum_{u\in B_{R}(p),v\in B_{R}(q)}G(x_{n},u)G(u,v;\Gamma \setminus (B_{R}(p)\cup B_{R}(q)))G(v,y_{n}).$$
By definition of convergence in the Martin compactification, for large enough $n$ we have for all $u\in B_{R}(p), v\in B_{R}(q)$,
$$(1-c)K_{\beta}(v)\leq \frac{G(v,y_{n})}{G(e,y_{n})}\leq (1+c)K_{\beta}(v),$$  
$$(1-c)\check{K}_{\alpha}(u)\leq \frac{G(x_{n},u)}{G(x_{n},e)}\leq (1+c)\check{K}_{\alpha}(u)$$
where $\check{K}$ denotes the Martin kernel for the reflected random walk driven by $\check{\mu}$, where $\check{\mu}(g)=\mu(g^{-1})$.
For simplicity, we define $\lambda_{\alpha,\beta}(c)$ by
$$1/\lambda_{\alpha,\beta}(c)=\sum_{u\in B_{R}(p),v\in B_{R}(q)}\check{K}_{\alpha}(u)K_{\beta}(v)G(u,v;\Gamma \setminus (B_{R}(p)\cup B_{R}(q))).$$
Notice that $\lambda_{\alpha,\beta}(c)$ does not depend on $n$.
We also write $\tilde{c}=\left (\frac{1+3c}{1-3c}\right )^{2}$ so that $\tilde{c}$ converges to 1 when $c$ tends to 0.
Consequently, we have for large enough $n$
$$G(x_{n},e)G(e,y_{n})\tilde{c}^{-1}\lambda_{\alpha,\beta}(c) \leq G(x_{n},y_{n})\leq G(x_{n},e)G(e,y_{n})\tilde{c}\lambda_{\alpha,\beta(c)}.$$
Similarly, we have
$$G(x'_{n},e)G(e,y'_{n})\tilde{c}^{-1}\lambda_{\alpha',\beta}(c) \leq G(x'_{n},y'_{n})\leq G(x'_{n},e)G(e,y'_{n})\tilde{c}\lambda_{\alpha',\beta}(c),$$
$$G(x'_{n},e)G(e,y_{n})\tilde{c}^{-1}\lambda_{\alpha',\beta}(c) \leq G(x'_{n},y_{n})\leq G(x'_{n},e)G(e,y_{n})\tilde{c}\lambda_{\alpha',\beta}(c),$$
$$G(x_{n},e)G(e,y'_{n})\tilde{c}^{-1}\lambda_{\alpha,\beta}(c) \leq G(x_{n},y'_{n})\leq G(x_{n},e)G(e,y'_{n})\tilde{c}\lambda_{\alpha,\beta}(c).$$

This implies that for all large enough $n$
$\frac{G(x_{n},y_{n})G(x'_{n},y'_{n})}{G(x'_{n},y_{n})G(x_{n},y'_{n})}$ is bounded between $(\frac{1-3c}{1+3c})^{8}$ and $(\frac{1+3c}{1-3c})^{8}$.
As $c$ can be chosen arbitrarily small, it follows that $$\frac{G(x_{n},y_{n})G(x'_{n},y'_{n})}{G(x'_{n},y_{n})G(x_{n},y'_{n})}\to 1.$$
This is a contradiction.
\end{proof}

\begin{corollary}\label{limitexists}
Let $\xi,\zeta \in \partial_{B}\Gamma$ be distinct conical points. Then the limit
$$\lim_{w\to \xi, v\to \zeta}\frac{G(w,v)}{G(w,e)G(e,v)}$$ exists and equals $\Theta(\xi,\zeta).$
\end{corollary}

\begin{proof}
By definition of $\Theta$, it suffices to show that the limit exists. 
To that end, suppose that $w_{n},w'_{n}\to \xi$ and $v_{n},v'_{n}\to \zeta$.
We need to show that for any such sequences we have
$$\frac{G(w_{n},v_{n})G(w'_{n},e)G(e,v'_{n})}{G(w'_{n},v'_{n})G(w_{n},e)G(e,v_{n})}\to 1.$$

Indeed, the quadruples $(w_{n},e,v_{n},v'_{n})$, $(w_{n},w'_{n},v_{n},e)$ and $(w_{n},w'_{n},v_{n},v'_{n})$ all satisfy the conditions of Proposition~\ref{ratiosconverge}.
Applying this proposition to each of these we thus obtain
$$\frac{G(w_{n},v_{n})G(e,v'_{n})}{G(e,v_{n})G(w_{n},v'_{n})}\to 1, \hspace{3pt}\frac{G(w_{n},v_{n})G(w'_{n},e)}{G(w'_{n},v_{n})G(w_{n},e)}\to 1, \hspace{3pt}\frac{G(w_{n},v_{n})G(w'_{n},v'_{n})}{G(w'_{n},v_{n})G(w_{n},v'_{n})}\to 1.$$
Multiplying the two first expressions and dividing by the last one, we get the desired convergence.
\end{proof}

We deduce the following.

\begin{proposition}\label{continuityNaim}
The Naim kernel is continuous at pairs of distinct conical points.
\end{proposition}

\begin{proof}
Let $\xi\neq \zeta$ be two conical limit points and let $\xi_n\to \xi$, $\zeta_n\to \zeta$ and $\xi_n\neq \zeta_n$.
According to Corollary~\ref{limitexists}, we can choose sequences $g_{n,m}$ and $h_{n,m}$ converging to $\xi_n$ and $\zeta_n$ respectively, as $m$ goes to infinity, such that
$$|\Theta(\xi_n,\zeta_n)-\Theta(g_{n,m},h_{n,m})|\leq \frac{1}{m}.$$
Since the Bowditch compactification is metrizable, we can thus construct sequences $g_n$ and $h_n$ with $g_n\to \xi$, $h_n\to \zeta$ such that $|\Theta(\xi_n,\zeta_n)-\Theta(g_n,h_n)|$ converges to 0.
Using Corollary~\ref{limitexists} again, $\Theta(g_n,h_n)$ converges to $\Theta(\xi,\zeta)$ so that $\Theta(\xi_n,\zeta_n)$ converges to $\Theta(\xi,\zeta)$.
\end{proof}

\subsection{Cross-ratios and stable translation length}\label{Sectioncrossratios}
We now prove some dynamical results about the Martin cocycle $c_M$ defined as
$$c_M(g,\xi)=-\mathrm{log}(K_{\xi}(g^{-1})),$$
where $K_{\xi}$ is the Martin kernel based at $\xi$ in the Martin boundary $\partial_{\mu}\Gamma$ and where $g\in \Gamma$.
The reason we call $c_M$ a cocycle is that it satisfies the following property, usually called the cocycle property (see \cite{GMM}).
\begin{equation}\label{cocycleproperty}
    \forall g,h\in \Gamma, \forall \xi\in \partial_{\mu}\Gamma, c_M(gh,\xi)=c_M(g,h\xi)+c_M(h,\xi).
\end{equation}
When $\xi \in \partial_B \Gamma$ is a conical limit point, we can identify it with a point of $\partial_{\mu}\Gamma$ and write 
$c_M(g,\xi)=-\mathrm{log}(K_{\xi}(g^{-1}))$ accordingly.

Recall that if $h,g\in \Gamma$ are any elements in any finitely generated group, endowed with a left-invariant metric $d$, the sequence
$d(h,g^nh)$ is sub-additive.
Hence, the sequence $\frac{d(h,g^nh)}{n}$ converges to $l(g)=\inf_{n}\frac{d(h,g^nh)}{n}$. It is easy to see that this limit is independent of $h$.
The number $l(g)$ is called the stable translation length.
By definition $l(g^n)=nl(g)$ for $n\in \mathbb{N}$.
It depends on the metric and in the following, we will compare different $l(g)$, for different metrics.
We fix the following notations.
We will denote by $l^w(g)$ the stable translation length associated to some word metric, and by $l^G(g)$ the stable translation length associated to the Green metric.

If $\Gamma$ is relatively hyperbolic, an infinite order element $g\in \Gamma$ is said to be loxodromic if it has two distinct fixed points $g^{+},g^{-}\in \partial_{B}\Gamma$. In that case we can order them so that $g^{n}x\to g^{+}$ for all $x\in \overline{\Gamma}_{B}\setminus \{g^{-}\}$ and $g^{-n}x\to g^{-}$ for all $x\in \overline{\Gamma}_{B}\setminus \{g^{+}\}$.
The points $g^-$ and $g^+$ are called the repelling and attracting fixed points of $g$ respectively, and are necessarily conical.  See e.g. \cite[Section 2]{Bowditch3} for details.

\begin{lemma}\label{lemmaMartincocycletranslationlength}
Let $\Gamma$ be a relatively hyperbolic group and let $g$ be a loxodromic element.
Denote by $g^-$ and $g^+$ the repelling and attractive fixed points of $g$.
Then, $c_M(g,g^+)=l^G(g)$ and $c_M(g,g^-)=-l^G(g)$. 
\end{lemma}

\begin{proof}
Using the cocycle property,
$c_M(g^n,g^+)=c_M(g^{n-1}, g^+)+c_M(g,g^+)$, so $c_M(g^n,g^+)=nc_M(g,g^+)$.
Since $g$ is loxodromic, for large enough $k$ and $n$, we have $\overline{\delta}^{f}_{e}(g^{n},g^{-k})>\overline{\delta}^{f}_{e}(g^{+},g^{-})/2>0$. Thus Proposition~\ref{Floydgeo} implies that $e$ lies a bounded (depending on $g$ but not on $k,n$) distance away from a transition point between $g^{-n}$ and $g^k$.
Relative Ancona inequalities~(\ref{relativeAncona}) then show that
$$G(g^{-n},g^k)\asymp G(g^{-n},e)G(e,g^k),$$ hence
$K(g^{-n},g^k)\asymp G(g^{-n},e)$.
Thus, $K(g^{-n},g^+)\asymp G(e,g^n)$.
The implied constants depend on $g$ but not on $k,n$.
In other words, 
$$\sup_{n}|c_M(g^n,g^+)-d_G(e,g^n)|<\infty,$$
so that
$\frac{1}{n}c_M(g^n,g^+)$ converges to $l^G(g)$.

For the second statement,
notice that for any $\xi$, $c_M(e,\xi)=0$.
Using again the cocycle property,
$c_M(e,g^-)=c_M(g,g^-)+c_M(g^{-1},g^-)$,
so $c_M(g,g^-)=-c_M(g^{-1},g^-)$.
Replacing $g$ with $g^{-1}$ in the first statement, we get $c_M(g,g^-)=-l^G(g^{-1})$.
Since $d_G(e,g^{n})=d_G(e,g^{-1})$, we finally get $l^G(g)=l^G(g^{-1})$.
\end{proof}

If $\xi_1,\xi_2,\xi_3,\xi_4$ are four distinct conical limit points, define their cross-ratio as
$$[\xi_1,\xi_2,\xi_3,\xi_4]=\frac{\Theta(\xi_1,\xi_3)}{\Theta(\xi_1,\xi_4)}\frac{\Theta(\xi_2,\xi_4)}{\Theta(\xi_2,\xi_3)},$$
where $\Theta$ is the Naïm Kernel.

Actually, $[\xi_1,\xi_2,\xi_3,\xi_4]$ is also well defined if $\xi_3=\xi_4$ or if $\xi_1=\xi_2$ and in any of these two cases, $[\xi_1,\xi_2,\xi_3,\xi_4]=1$.
Denote by $\partial \hat{\Gamma}$ the set of conical limit points and by $\mathfrak{C}$ the set on which the cross-ratio is well defined, that is
$$\mathfrak{C}:=\{\xi_1,\xi_2,\xi_3,\xi_4\in \partial \hat{\Gamma},\xi_1\neq \xi_3,\xi_1\neq \xi_4,\xi_2\neq \xi_3,\xi_2\neq \xi_4\}.$$
The cross-ratio is then well defined on $\mathfrak{C}$.
It is also continuous on $\mathfrak{C}$ according to Proposition~\ref{continuityNaim}.
We have the following two lemmas, adapted from Dal'Bo \cite{DalBo} and Otal \cite{Otal}.

\begin{lemma}\label{lemmadynamic1}
Let $g$ be a loxodromic element of $\Gamma$ and denote by $g^{-},g^+$ its repelling and attractive fixed points.
Let $\xi$ be a conical limit point that differ from $g^-$ and $g^+$.
Then,
$\mathrm{e}^{-2l^G(g)}=[g^-,g^+,g\cdot \xi,\xi]$.
\end{lemma}

\begin{proof}
By definition, $\mathrm{log}([g^-,g^+,g\cdot \xi, \xi])$ is the limit, when $g_n$ converges to $g^-$ and $h_n$ converges to $g^+$, of the quantity
$$\mathrm{log}(K_{g\cdot \xi}(g_n))+\mathrm{log}(K_{\xi}(h_n))-\mathrm{log}(K_{g\cdot \xi}(h_n))-\mathrm{log}(K_{\xi}(g_n)).$$
Choosing the sequences $g_n=g^{-n}$ and $h_n=g^n$, we then have
\begin{equation}\label{crossratioformula}
\begin{split}
\mathrm{log}([g^-,g^+,g\cdot \xi, \xi])=\lim_{n\to \infty}-&c_M(g^n,g\cdot \xi)-c_M(g^{-n},\xi)\\&+c_M(g^{-n},g\cdot \xi)+c_M(g^n,\xi).
\end{split}
\end{equation}
Using the cocycle property~(\ref{cocycleproperty}), an easy induction shows that for every $h\in \Gamma$ and every conical limit point $\zeta$,
$c_M(h^n,\zeta)=\sum_{k=0}^{n-1}c_M(h,h^k\cdot \zeta)$.
Thus,~(\ref{crossratioformula}) can be reformulated as
\begin{equation}\label{crossratioformula2}
\begin{split}
\mathrm{log}([g^-,g^+,g\cdot \xi, \xi])=\lim_{n\to \infty}&-c_M(g,g^n\cdot \xi)+c_M(g,\xi)\\&+c_M(g^{-1},g\cdot \xi)-c_M(g^{-1},g^{-n+1}\cdot \xi).
\end{split}
\end{equation}
Using the cocycle property~(\ref{cocycleproperty}) again,
$c_M(g,\xi)+c_M(g^{-1},g\cdot \xi)=c_M(e,\xi)=0$.
Since $g^n\cdot \xi$ converges to $g^+$ and $g^{-n+1}\cdot \xi$ converges to $g^-$, by continuity of the Martin kernels, Expression~(\ref{crossratioformula2}) becomes
\begin{align*}
    \mathrm{log}([g^-,g^+,g\cdot \xi, \xi])&=\lim_{n\to \infty}-c_M(g,g^n\cdot \xi)-c_M(g^{-1},g^{-n+1}\cdot \xi)\\&=-c_M(g,g^+)-c_M(g^{-1},g^-).
\end{align*}
Finally, by Lemma~\ref{lemmaMartincocycletranslationlength}, $c_M(g,g^+)=c_M(g^{-1},g^-)=l^G(g)$.
\end{proof}

\begin{lemma}\label{lemmadynamic2}
Let $g$ and $h$ be two independent loxodromic elements of $\Gamma$.
Denote by $g^{-}$ and $g^{+}$, respectively $h^{-}$, $h^{+}$ the repelling and attractive fixed points of $g$, respectively $h$.
Then,
$$[g^{-},h^{-},g^+,h^+]=\lim_{n\to +\infty}\mathrm{e}^{(l^G(g^nh^n)-l^G(g^n)-l^G(h^n))}.$$
\end{lemma}

\begin{proof}
By definition, $\mathrm{log}([g^-,h^-,g^+,h^+])$ is the limit, when $g_n$ converges to $g^-$ and $h_n$ converges to $h^-$, of the quantity
$$\mathrm{log}(K_{g^+}(g_n))+\mathrm{log}(K_{h^+}(h_n))-\mathrm{log}(K_{g^+}(h_n))-\mathrm{log}(K_{h^+}(g_n)).$$
As in the previous proof, choosing the sequences $g_n=g^{-n}$ and $h_n=h^{-n}$, we have
\begin{equation}\label{crossratioformula3}
\begin{split}
\mathrm{log}([g^{-},h^{-},g^+,h^+])=\lim_{n\to \infty}-c_M(g^n,g^+)&-c_M(h^n,h^+)\\&+c_M(h^{n},g^+)+c_M(g^n,h^+).
\end{split}
\end{equation}
For large $n$, $\xi_n:=g^nh^n$ is loxodromic.
Denote by $\xi_n^-$ and $\xi_n^+$ its repelling and attractive limit points.
According to Lemma~\ref{lemmaMartincocycletranslationlength},
$l^G(g^nh^n)=c_M(g^nh^n,\xi_n^+)$ and by the cocycle property~(\ref{cocycleproperty}),
$l^G(g^nh^n)=c_M(g^n,h^n\cdot \xi_n^+)+c_M(h^n,\xi_n^+)$.
Similarly, we have $l^G(g^n)=c_M(g^n,g^+)$ and $l^G(h^n)=c_M(h^n,h^+)$.
Thus, using~(\ref{crossratioformula3}), we only have to prove that
$$c_M(h^n,g^+)+c_M(g^n,h^+)-c_M(g^n,h^n\cdot \xi_n^+)-c_M(h^n,\xi_n^+)\underset{n\to \infty}{\longrightarrow}0.$$
Since $\xi_n=g^nh^n$, $\xi_n^+$ converges to $g^+$.
Moreover, $h^n\cdot\xi_n^+$ converges to $h^+$.
Since $g$ and $h$ are independent, for large enough $n$, $g^-$ is bounded away from $h^+$ and $h^n\cdot \xi_n^+$.
Finally, since $g^{-n}$ converges to $g^-$,
$$\lim_{n\to \infty}\frac{K_{h^+}(g^{-n})}{K_{h^n\cdot \xi_n^+}(g^{-n})}=\lim_{n\to \infty}\frac{\Theta (g^-,h^+)}{\Theta (g^-,h^n\cdot \xi_n^+)}=1.$$
Similarly,
$$\lim_{n\to \infty}\frac{K_{g^+}(h^{-n})}{K_{\xi_n^+}(h^{-n})}=1.$$
In other words, $c_M(g^n,h^+)-c_M(g^n,h^n\cdot \xi_n^+)$ and $c_M(h^n,g^+)-c_M(h^n,\xi_n^+)$ both converge to 0.
This concludes the proof.
\end{proof}

\subsection{Arithmeticity of the stable translation lengths}\label{Sectionarithmetictranslationlength}
Let $\Gamma$ be a relatively hyperbolic group endowed with a word metric.
We prove here that the stable translation length of loxodromic elements is arithmetic.
It thus takes value in a discrete space.
We then deduce the same thing for the Green metric whenever $h=lv$.

\begin{proposition}\label{proparithmeticityword}
Let $\Gamma$ be a relatively hyperbolic groups, endowed with a word metric. There exists some rational number $\alpha_w$ such that
for any loxodromic element $g$, $l^w(g)\in \alpha_w \N$.
Furthermore, $l^{w}(g)>0$ for any loxodromic element $g$.
\end{proposition}

\begin{proof}
This result was first stated by Gromov in \cite{Gromov} for hyperbolic groups.
We will adapt the proof of Delzant \cite{Delzant} to the case of relatively hyperbolic groups.

We fix a finite generating set $S$ and arbitrarily choose some order on $S$.
Between two elements $g_1,g_2$, there is a unique geodesic, ordered lexicographically according to the order chosen on $S$.
We say that such a geodesic is special.
If $\xi_1$ and $\xi_2$ are two conical limit points, we say that a geodesic between $\xi_1$ and $\xi_2$ is special if every finite sub-path of it is special.
Such a geodesic always exists, as we now prove.
Consider first a (non-necessarily special) geodesic $(\xi_1,\xi_2)$ between $\xi_1$ and $\xi_2$ and let $g_n,h_n$ be two sequences of points on $(\xi_1,\xi_2)$, with $g_n\to \xi_1$, $h_n\to \xi_2$.
Consider then a special geodesic from $\alpha_n$ between $g_n$ and $h_n$.
Fixing a transition point $g$ on $(\xi_1,\xi_2)$, for large enough $n$, all the geodesics $\alpha_n$ pass within a bounded distance of $g$, so that up to extracting a sub-sequence and up to changing $g$, we can assume that $g$ lies on all the $\alpha_n$.
We can then extract a sub-sequence of the $\alpha_n$ that converges to a new geodesic from $\xi_1$ to $\xi_2$, which is necessarily special.

Now let $g$ be a loxodromic element and fix a special geodesic $\alpha$ between $g^-$ and $g^+$.
Since $g$ is loxodromic, $g^-$ and $g^+$ are conical. Thus, for suitable $\epsilon,\eta>0$, Lemma~\ref{transitionpoints}~c) implies there are an infinite number of $(\epsilon,\eta)$-transition points on both sides of $\alpha$.
In particular, there is a $\Z$-indexed sequence $...,h_{-n},...,h_0,...,h_n...$ such that $h_k$ is a $(\epsilon,\eta)$-transition point on $\alpha$ and $h_{-n}\to g^-$ and $h_n\to g^+$.

Furthermore,  by Proposition~\ref{Floydgeo} there exists a constant $R>0$ such that if $h$ is an $(\epsilon,\eta)$-transition point on $\alpha$, then any geodesic from $g^-$ to $g^+$ goes through a point at distance at most $R$ from $h$. 
(The constants $\epsilon,\eta,R$ depend only on the Cayley graph, and are independent of $g$).

Denote by $N$ the maximal cardinality of a ball of radius $R$ in the Cayley graph.
Then, if $k_1\neq k_2\in \Z$, there are at most $N^2$ special geodesics from an element of $B(h_{k_1},R)$ to an element of $B(h_{k_2},R)$.
This proves that there are at most $N^2$ special geodesics from $g^-$ to $g^+$.
Now, $g$ permutes those special geodesic, so that, denoting $K=N^2!$, $g^K$ fixes one of them.
Fix any point $h$ on this special geodesic, so that $d((g^K)^nh,h)=nd(g^Kh,h)$.
This shows that $l^w(g^K)\in \N$.
Since $l^w(g^K)=Kl^w(g)$, we see that $l^w(g)$ lies in $\frac{1}{K}\N$ and $K$ does not depend on $g$.
Furthermore, we have that $d((g^K)^nh,h)/n=d(g^Kh,h)>0$ for all $n$, so that $l^w(g^K)>0$.
It follows that $l^w(g)=l^w(g^K)/K>0$, which concludes the proof.
\end{proof}

We deduce the following from Theorem~\ref{theoremBHM}.

\begin{proposition}\label{proparithmeticityGreen}
Let $\Gamma$ be a non-elementary relatively hyperbolic group.
Let $\mu$ be an admissible probability measure on $\Gamma$ with finite super-exponential moment.
Fix a finite generating set $S$ and consider the associated word metric.
Assume that $h=lv$.
Then, there exists a real number $\alpha_G$ such that for any loxodromic element $g$, $l^G(g)\in \alpha_G \N$.
\end{proposition}

\begin{proof}
Since $h=lv$, Theorem~\ref{theoremBHM} shows that $|d_G-vd_w|\leq C$, where $d_G$ is the Green metric and $d_w$ the word metric and where $C$ is some constant.
Thus, the limits of $\frac{d_G(e,\gamma^n)}{n}$ and  $v\frac{d_w(e,\gamma^n)}{n}$ are the same, i.e. $l^G(g)=vl^w(g)$.
According to Proposition~\ref{proparithmeticityword}, for any loxodromic element $g$, $l^w(g)\in \alpha_w\N$, so $l^G(g)\in v\alpha_w\N$.
\end{proof}

\subsection{Connectedness of the boundaries}\label{Sectionconnectedboundaries}
We can now prove the following result.

\begin{theorem}\label{mainTheoremconnected}
Let $\Gamma$ be a non-elementary relatively hyperbolic group.
Denote by $\partial\hat{\Gamma}$ the set of conical limit points and
assume that for any two distinct conical limit points $\xi_1,\xi_2$, $\partial\hat{\Gamma}\setminus \{\xi_1,\xi_2\}$ has a finite number of connected components.
Let $\mu$ be an admissible probability measure on $\Gamma$ with finite super-exponential moment.
Then, for any word metric associated to a finite generating set, we have $h<lv$.
\end{theorem}

\begin{proof}
Assume on the contrary that $h=lv$.
Lemma~\ref{lemmadynamic2} shows that for any loxodromic elements $g$ and $h$,
$$[g^{-},h^{-},g^+,h^+]=\lim_{n\to +\infty}\mathrm{e}^{\frac{1}{2}(l^G(g^nh^n)-l^G(g^n)-l^G(h^n))}.$$
Using Proposition~\ref{proparithmeticityGreen}, we see that
$[g^{-},h^{-},g^+,h^+]$ lies in a discrete set $\mathrm{e}^{\alpha_G\N}$.
Since the set of $\{(g^-,g^+)\}$ is dense in the set $\{(\xi_1,\xi_2),\xi_1\neq \xi_2\}$,
using the continuity of the cross-ratio, we have that for any four conical limit points $\xi_1,\xi_2,\xi_3,\xi_4$ in $\mathfrak{C}$,
$[\xi_1,\xi_2,\xi_3,\xi_4]$ lies in $\mathrm{e}^{\alpha_G\N}$.

Let $g$ be a loxodromic element and
fix any conical limit point $\xi_0$ that differs from $g^-$ and $g^+$.
Then, $g\cdot \xi_0$ and $g^2\cdot \xi_0$ also differ from $g^-$ and $g^+$.
Denote by $\mathcal{C}$ the connected component of $\partial \hat{\Gamma}\setminus \{g^-,g^+\}$ in which $\xi_0$ lie.
By assumption there is a finite number of components in $\partial \hat{\Gamma}\setminus \{g^-,g^+\}$ and since $g$ permutes these components, there is a power of $g$, $g^{n_0}$ such that $\mathcal{C}$ is invariant by $g^{n_0}$.
Replacing $g$ by $g^{n_0}$, we can assume that $n_0=1$.

Since $g\cdot \xi_0\in \mathcal{C}$, $g^{\pm}\neq g\cdot \xi_0$.
Similarly, if $\xi\in \mathcal{C}$, $g^{\pm}\neq \xi$ and so the function $\xi \in \mathcal{C}\mapsto [g^-,g^+,g\cdot\xi_0,\xi]$ is continuous on $\mathcal{C}$ and has its image contained in $\mathrm{e}^{\alpha_G\N}$, so it is constant.
However, according to Lemma~\ref{lemmadynamic1}, $[g^-,g^+,g\cdot\xi_0,\xi_0]=\mathrm{e}^{-2l^G(g)}$ and $l^G(g)=vl^w(g)>0$, so that $[g^-,g^+,g\cdot\xi_0,\xi_0]<1$.
On the other hand, $[g^-,g^+,g\cdot\xi_0,g\cdot \xi_0]=1$.
This is a contradiction.
\end{proof}

Theorem~\ref{sphericalBowditchboundary} immediately follows from Theorem~\ref{mainTheoremconnected}.
Indeed, if the Bowditch boundary is a sphere of dimension $d\geq 2$, then the set of conical limit points with two points removed is a sphere with a finite or countable number of points removed.
Such a space is connected as we now prove.
It is enough to prove that $\R^d\setminus A$ is connected whenever $d\geq 2$ and $A$ is countable.
Consider two points $x,y\in \R^d\setminus A$.
There exists a straight line passing through $x$ that does not pass through any of the points in $A$, since there is an uncountable set of distinct such straight lines.
Similarly, there exists a straight line passing through $y$ that does not pass through any points in $A$ and that is not parallel to the first straight line.
Those two lines necessarily intersect, which yields a continuous path from $x$ to $y$. \qed

\begin{remark}
In Corollary~\ref{corofundamentalgroups}, if the manifold $M$ is of dimension $n=2$ and is non-compact, then the set of conical limit points is a circle $\mathbb{S}^1$ with a countable number of points removed and we cannot apply Theorem~\ref{mainTheoremconnected}.
Actually, in this setting, $\pi_1(M)$ is free and so $h=lv$ can occur.
However, if $M$ is compact, then the set of conical limit points is homeomorphic to $\mathbb{S}^1$ and so we can apply Theorem~\ref{mainTheoremconnected}. We thus recover the fact that $h<lv$ for the word metric on a surface group, which is already a consequence of results in \cite{GMM}.
\end{remark}


\bibliographystyle{plain}
\bibliography{hlv}

\begin{thebibliography}{10}

\bibitem{Ancona}
Alano Ancona.
\newblock Positive harmonic functions and hyperbolicity.
\newblock In {\em Potential theory-surveys and problems}, pages 1--23. Lecture
  notes in mathematics, Springer, 1988.

\bibitem{BlachereBrofferio}
S\'ebastien Blach\`ere and Sara Brofferio.
\newblock Internal diffusion limited aggregation on discrete groups having
  exponential growth.
\newblock {\em Probabability Theory and Related Fields}, 137:323–343, 2007.

\bibitem{BlachereHassinskyMathieu1}
S\'ebastien Blach\`ere, Peter Ha\"issinsky, and Pierre Mathieu.
\newblock Asymptotic entropy and {G}reen speed for random walks on countable
  groups.
\newblock {\em The Annals of Probability}, 36:1134–1152, 2008.

\bibitem{BlachereHassinskyMathieu2}
S\'ebastien Blach\`ere, Peter Ha\"issinsky, and Pierre Mathieu.
\newblock Harmonic measures versus quasiconformal measures for hyperbolic
  groups.
\newblock {\em Annales Scientifiques de l’\'Ecole Normale Sup\'erieure},
  44:683–721, 2011.

\bibitem{Bowditch2}
Brian Bowditch.
\newblock Geometrical finiteness with variable negative curvature.
\newblock {\em Duke Mathematical Journal}, 77:229--274, 1995.

\bibitem{Bowditch3}
Brian Bowditch.
\newblock Convergence groups and configuration spaces.
\newblock In {\em Group theory down under}, pages 23--54. Walter de Gruyter,
  Berlin, 1999.

\bibitem{Bowditch}
Brian Bowditch.
\newblock Relatively hyperbolic group.
\newblock {\em International Journal of Algebra and Computation}, 22:66 pp,
  2012.

\bibitem{CandelleroGilchMuller}
Elisabetta Candellero, Lorenz Gilch, and Sebastian M\"uller.
\newblock Branching random walks on free products of groups.
\newblock {\em Proceedings of the London Mathematical Society}, 104:1085--1120,
  2012.

\bibitem{Coornaert}
Michel Coornaert.
\newblock Mesures de {P}atterson-{S}ullivan sur le bord d'un espace
  hyperbolique au sens de {G}romov ({F}rench).
\newblock {\em Pacific Journal of Mathematics}, 159:241--270, 1993.

\bibitem{DalBo}
Fran\c{c}oise Dal'Bo.
\newblock Remarques sur le spectre des longueurs d'une surface et comptages
  ({F}rench).
\newblock {\em Boletim da Sociedade Brasileira de Matem\'atica}, 30:199–221,
  1999.

\bibitem{Delzant}
Thomas Delzant.
\newblock Sous-groupes distingu\'es et quotients des groupes hyperboliques
  ({F}rench).
\newblock {\em Duke Mathematical Journal}, 83:661--682, 1996.

\bibitem{Derriennic}
Yves Derriennic.
\newblock Entropie, th\'eor\`emes limites et marches al\'eatoires ({F}rench).
\newblock In {\em Probability Measures on Groups VIII}, pages 241--284. Lecture
  notes in mathematics, Springer, 1986.

\bibitem{DrutuSapir}
Cornelia Dru\c{t}u and Mark Sapir.
\newblock Tree graded spaces and asymptotic cones of groups.
\newblock {\em Topology}, 44:959–1058, 2005.
\newblock with an {A}ppendix by {D}enis {O}sin and {M}ark {S}apir.

\bibitem{Dussaule}
Matthieu Dussaule.
\newblock The {M}artin boundary of a free product of abelian groups.
\newblock arXiv:1709.07738, to appear in Annales de l'Institut Fourier, 2017.

\bibitem{DGGP}
Matthieu Dussaule, Ilya Gekhtman, Victor Gerasimov, and Leonid Potyagailo.
\newblock The {M}artin boundary of relatively hyperbolic groups with virtually
  abelian parabolic subgroups.
\newblock arXiv:1711.11307, 2017.

\bibitem{Farb}
Benson Farb.
\newblock Relatively hyperbolic groups.
\newblock {\em Geometric and Functional Analysis}, 8:810--840, 1998.

\bibitem{Federer}
Herbert Federer.
\newblock {\em Geometric Measure Theory}.
\newblock Classics in Mathematics, Springer, 1996.

\bibitem{Furman}
A.~Furman.
\newblock Coarse-geometric perspective on negatively curved manifolds and
  groups.
\newblock In M.~Burger and A.~Iozzi, editors, {\em Rigidity in Dynamics and
  Geometry}, pages 149--166. Springer, 2002.

\bibitem{GGPY}
Ilya Gekhtman, Victor Gerasimov, Leonid Potyagailo, and Wenyuan Yang.
\newblock Martin boundary covers {F}loyd boundary.
\newblock arXiv:1708.02133, 2017.

\bibitem{GTT}
Ilya Gekhtman, Samuel Taylor, and Giulio Tiozzo.
\newblock Counting problems in graph products and relatively hyperbolic groups.
\newblock arXiv:1711.04177, 2017.

\bibitem{GTgibbs}
Ilya Gekhtman and Giulio Tiozzo.
\newblock Stationary measures vs {G}ibbs measures for geometrically finite
  actions on {C}{A}{T}(-1) spaces.
\newblock arXiv:1904.01187, 2019.

\bibitem{Gerasimov2}
Victor Gerasimov.
\newblock Expansive convergence groups are relatively hyperbolic.
\newblock {\em Geometric and Functional Analysis}, 19:137--169, 2009.

\bibitem{Gerasimov}
Victor Gerasimov.
\newblock Floyd maps for relatively hyperbolic groups.
\newblock {\em Geometric and Functional Analysis}, 22:1361--1399, 2012.

\bibitem{GePoJEMS}
Victor Gerasimov and Leonid Potyagailo.
\newblock Quasi-isometries and {F}loyd boundaries of relatively hyperbolic
  groups.
\newblock {\em Journal of the European Mathematical Society}, 15:2115--2137,
  2013.

\bibitem{GerasimovPotyagailo}
Victor Gerasimov and Leonid Potyagailo.
\newblock Non-finitely generated relatively hyperbolic groups and {F}loyd
  quasiconvexity.
\newblock {\em Groups, Geometry, and Dynamics}, 9:369--434, 2015.

\bibitem{GePoCrelle}
Victor Gerasimov and Leonid Potyagailo.
\newblock Quasiconvexity in the relatively hyperbolic groups.
\newblock {\em Journal for Pure and Applied Mathematics}, 710:95--135, 2016.

\bibitem{GMM}
S\'ebastien Gou\"ezel, Fr\'ed\'eric Math\'eus, and Fran\c{c}ois Maucourant.
\newblock Entropy and drift in word hyperbolic groups.
\newblock {\em Inventiones Mathematicae}, 211:1201--1255, 2018.

\bibitem{Gromov}
Mikhael Gromov.
\newblock Hyperbolic groups.
\newblock In S.M. Gersten, editor, {\em Essays in group theory}, pages 75--265.
  M.S.R.I. publications 8, Springer, 1987.

\bibitem{Guivarch}
Yves Guivarc'h.
\newblock Sur la loi des grands nombres et le rayon spectral d’une marche
  aléatoire ({F}rench).
\newblock In {\em Conference on Random Walks}, pages 47--98. Astérisque, vol.
  74, Soc. Math. France, 1980.

\bibitem{Haissinsky}
Peter Haïssinsky.
\newblock Marches aléatoires sur les groupes hyperboliques ({F}rench).
\newblock Course notes from the conference "Regards croisés sur les marches
  aléatoires et la géométrie des groupes, en l'honneur d'Emile Le Page ",
  http://web.univ-ubs.fr/lmam/matheus/emilfest.html, 2011.

\bibitem{Hruska}
Geoffrey Hruska.
\newblock Relative hyperbolicity and relative quasiconvexity for countable
  groups.
\newblock {\em Algebraic and Geometric Topology}, 10:1807–1856, 2010.

\bibitem{Kaimanovich2}
Vadim Kaimanovich.
\newblock Boundaries of invariant {M}arkov operators: the identification
  problem.
\newblock In {\em Ergodic theory of $\mathbb{Z}^d$ actions}, pages 127--176.
  London Mathematical Society Lecture Note series, vol. 228, Cambridge
  University Press, 1996.

\bibitem{KaimanovichVershik}
Vadim Kaimanovich and Anatoly Vershik.
\newblock Random walks on discrete groups: boundary and entropy.
\newblock {\em Annals of Probability}, 11:457--490, 1983.

\bibitem{Kaimanovich}
Vadim~A. Kaimanovich.
\newblock The {P}oisson formula for groups with hyperbolic properties.
\newblock {\em Annals of Mathematics}, 152:659--692, 2000.

\bibitem{Karlsson2}
Anders Karlsson.
\newblock Boundaries and random walks on finitely generated infinite groups.
\newblock {\em Arkiv f\"or Matematik}, 41:295--306, 2003.

\bibitem{Karlsson}
Anders Karlsson.
\newblock Free subgroups of groups with nontrivial {F}loyd boundary.
\newblock {\em Communications in Algebra}, 31:5361--5376, 2003.

\bibitem{Ledrappier}
Fran\c{c}ois Ledrappier.
\newblock Some asymptotic properties of random walks in free groups.
\newblock In {\em Topics in Probability and Lie Groups: Boundary Theory}, pages
  117--152. CRM Proceedings and Lecture Notes, 2001.

\bibitem{MaherTiozzo}
Joseph Maher and Giulio Tiozzo.
\newblock Random walks on weakly hyperbolic groups.
\newblock {\em Journal f\"{u}r die reine und angewandte Mathematik},
  742:187--239, 2018.

\bibitem{Malcev}
Anatoly Mal'cev.
\newblock On a class of homogeneous spaces.
\newblock {\em American Mathematical Society Translations}, 39, 1951.

\bibitem{Mathieu-Sisto}
Pierre Mathieu and Alessandro Sisto.
\newblock Deviation inequalities and {C}{L}{T} for random walks on
  acylindrically hyperbolic groups.
\newblock arXiv:1411.7865, 2014.

\bibitem{MYJ}
Katsuhiko Matsuzaki, Yasuhiro Yabuki, and Johannes Jaerisch.
\newblock Normalizer, divergence type and {P}atterson measure for discrete
  groups of the {G}romov hyperbolic space.
\newblock arXiv:1511.02664, 2015.

\bibitem{NicaSpakula}
Bogdan Nica and J\'an \v{S}pakula.
\newblock Strong hyperbolicity.
\newblock arXiv:1408.0250, 2014.

\bibitem{Osin}
Denis Osin.
\newblock Acylindrically hyperbolic groups.
\newblock {\em Transactions of the American Mathematical Society},
  368:851--888, 2016.

\bibitem{Otal}
Jean-Pierre Otal.
\newblock Sur la g\'eom\'etrie symplectique de l'espace des g\'eod\'esiques
  d'une vari\'et\'e \`a courbure n\'egative ({F}rench).
\newblock {\em Revista matematica Ibero americana}, 8, 1992.

\bibitem{Pansu}
Pierre Pansu.
\newblock Croissance des boules et des g\'eod\'esiques ferm\'ees dans les
  nivari\'et\'es ({F}rench).
\newblock {\em Ergodic Theory and Dynamical Systems}, 3:415--445, 1983.

\bibitem{PotyagailoYang}
Leonid Potyagailo and Wenyuan Yang.
\newblock Hausdorff dimension of boundaries of relatively hyperbolic groups.
\newblock arXiv:1609.01763, 2016.

\bibitem{Raghunathan}
Madabusi~Santanam Raghunathan.
\newblock {\em Discrete subgroups of {L}ie groups}.
\newblock Springer, 1972.

\bibitem{Sawyer}
Stanley Sawyer.
\newblock Martin boundaries and random walks.
\newblock {\em Contemporary mathematics}, 206:17--44, 1997.

\bibitem{Sisto}
Alessandro Sisto.
\newblock On metric relative hyperbolicity.
\newblock arXiv:1210.8081, 2012.

\bibitem{Tanaka}
Ryokichi Tanaka.
\newblock Dimension of harmonic measures in hyperbolic spaces.
\newblock {\em Ergodic Theory and Dynamical Systems}, 39:474--499, 2019.

\bibitem{Tukia}
Pekka Tukia.
\newblock Conical limit points and uniform convergence groups.
\newblock {\em Journal for Pure and Applied Mathematics}, 501:71--98, 1998.

\bibitem{Vershik}
Anatoly Vershik.
\newblock Dynamic theory of growth in groups: entropy, boundaries, examples.
\newblock {\em Uspekhi Matematicheskikh Nauk}, 55:59--128, 2000.

\bibitem{Walsh}
Cormac Walsh.
\newblock The action of a nilpotent group on its horofunction boundary has
  finite orbits.
\newblock {\em Groups Geometry, and Dynamics}, 5:189--206, 2011.

\bibitem{Walters}
Peter Walters.
\newblock {\em An Introduction to Ergodic Theory}.
\newblock Springer, 1982.

\bibitem{Woess}
Wolfgang Woess.
\newblock {\em Random Walks on Infinite Graphs and Groups}.
\newblock Cambridge Press University, 2000.

\bibitem{Yaman}
Asli Yaman.
\newblock A topological characterisation of relatively hyperbolic groups.
\newblock {\em Journal for Pure and Applied Mathematics}, 566:41--89, 2004.

\bibitem{Yang}
Wenyuan Yang.
\newblock Patterson-{S}ullivan measures and growth of relatively hyperbolic
  groups.
\newblock arXiv:1308.6326, 2013.

\end{thebibliography}

\end{document}